\documentclass[11pt, reqno]{amsart}
\usepackage{amsrefs, amsmath, amssymb, amsfonts, amstext, verbatim, amsthm}
\usepackage{graphicx}
\usepackage[all]{xy}
\usepackage{color}
\usepackage{enumitem}
\usepackage{tikz}
\usepackage{caption,subcaption}
\usepackage{float}
\usepackage{epsfig}

\usepackage[colorlinks=true, linkcolor=blue, citecolor=red]{hyper ref}

\usepackage{fullpage}

\theoremstyle{plain}
\newtheorem{thm}{Theorem}[section]
\newtheorem{lem}[thm]{Lemma}

\newtheorem{cor}[thm]{Corollary}

\newtheorem{prop}[thm]{Proposition}

\newtheorem{observation}[thm]{Observation}
\newtheorem*{thm*}{Theorem}

\newtheorem{sumthm}[thm]{Summary Theorem}

\theoremstyle{definition}
\newtheorem{rem}[thm]{Remark}
\newtheorem{defn}[thm]{Definition}
\newtheorem{ex}[thm]{Example}

{\end{minipage} \end{center}}

\newcommand{\bA}{\mathbf{A}}
\newcommand{\cA}{\mathcal{A}}

\newcommand{\cB}{\mathcal{B}}

\newcommand{\bC}{\mathbf{C}}
\newcommand{\cC}{\mathcal{C}}

\newcommand{\bD}{\mathbf{D}}

\newcommand{\cE}{\mathcal{E}}

\newcommand{\cF}{\mathcal{F}}

\newcommand{\cG}{\mathcal{G}}

\newcommand{\cH}{\mathcal{H}}

\newcommand{\cK}{\mathcal{K}}

\newcommand{\bL}{\mathbf{L}}
\newcommand{\cL}{\mathcal{L}}

\newcommand{\cM}{\mathcal{M}}

\newcommand{\bN}{\mathbf{N}}

\newcommand{\cO}{\mathcal{O}}

\newcommand{\bP}{\mathbf{P}}
\newcommand{\cP}{\mathcal{P}}

\newcommand{\bR}{\mathbf{R}}

\newcommand{\cV}{\mathcal{V}}

\newcommand{\fX}{\mathfrak{X}}

\newcommand{\bZ}{\mathbf{Z}}

\newcommand{\id}{\mathop{{\rm id}}\nolimits}

\renewcommand{\bR}{{\mathbb R}}
\renewcommand{\bC}{{\mathbb C}}
\renewcommand{\bZ}{{\mathbb Z}}

\newcommand{\op}[1]{\!\!\mathop{\rm ~#1}\nolimits}

\renewcommand{\bP}{\mathbb{P}}
\renewcommand{\bA}{\mathbb{A}}
\newcommand{\pt}{\ast}

\newcommand{\lie}[1]{\mathfrak{#1}}

\newcommand{\inner}[1]{{\underline{#1}}}

\newcommand{\dual}{\vee}

\newcommand{\X}{\fX}

\newcommand{\ladj}[1]{\beta_{< #1}}
\newcommand{\G}{\cG}
\newcommand{\D}{\op{D}}
\newcommand{\bdot}{{\begin{picture}(4,4)\put(2,3){\circle*{1.5}}\end{picture}}}

\newcommand{\parr}{\dashrightarrow}

\def\lmut{\mathbb{L}}
\def\rmut{\mathbb{R}}

\def\tensor{\otimes}

\DeclareMathOperator{\SL}{SL}
\DeclareMathOperator{\Perf}{\mathfrak{Perf}}
\def\blt{\bdot}
\newcommand{\sod}[1]{\langle #1 \rangle}

\begin{document}

\title{Autoequivalences of derived categories via geometric invariant theory}

%\author[1]{Daniel Halpern-Leistner\thanks{danielhl@math.berkeley.edu}}
%\author[1]{Author B\thanks{ishipman@umich.edu}}
%\author[2]{Author C\thanks{C.C@university.edu}}
%\affil[1]{Department of Mathematics, UC Berkeley}
%\affil[2]{Department of Mathematics, University of Michigan}
%
%%\renewcommand\Authands{ and }

%\author{Daniel Halpern-Leistner (corresponding author) \\ \texttt{Mathematics Department, UC Berkeley, 970 Evans Hall \# 3840, Berkeley, CA 94720-3840 USA} \\ \texttt{danielhl@math.berkeley.edu} \\ \texttt{001-215-435-0563} \\ \\ Ian Shipman \\ \texttt{Mathematics Department, University of Michigan} \\ \texttt{ishipman@umich.edu}}

\author{Daniel Halpern-Leistner, Ian Shipman}

\begin{abstract}
We study autoequivalences of the derived category of coherent sheaves of a variety arising from a variation of GIT quotient. We show that these autoequivalences are spherical twists, and describe how they result from mutations of semiorthogonal decompositions. Beyond the GIT setting, we show that all spherical twist autoequivalences of a dg-category can be obtained from mutation in this manner.

Motivated by a prediction from mirror symmetry, we refine the recent notion of ``grade restriction rules" in equivariant derived categories. We produce additional derived autoequivalences of a GIT quotient and propose an interpretation in terms of monodromy of the quantum connection. We generalize this observation by proving a criterion under which a spherical twist autoequivalence factors into a composition of other spherical twists.

\end{abstract}

\maketitle

%\linenumbers

\tableofcontents

\section{Introduction}

Homological mirror symmetry predicts, in certain cases, that the bounded derived category of coherent sheaves on an algebraic variety should admit \emph{twist autoequivalences} corresponding to a spherical object \cite{ST01}. The autoequivalences predicted by mirror symmetry have been widely studied, and the notion of a spherical object has been generalized to the notion of a spherical functor \cite{A} (See Definition \ref{def:spherical_functor}). We apply recently developed techniques for studying the derived category of a geometric invariant theory (GIT) quotient \cite{BFK,DS,HL12,HHP09,Se} to the construction of autoequivalences, and our investigation leads to general connections between the theory of spherical functors and the theory of semiorthogonal decompositions and mutations.

We consider an algebraic stack which arises as a GIT quotient of a smooth quasiprojective variety $X$ by a reductive group $G$. By varying the $G$-ample line bundle used to define the semistable locus, one gets a birational transformation $X^{ss}_- / G \dashrightarrow X^{ss}_+ /G$ called a variation of GIT quotient (VGIT). We study a simple type of VGIT, which we call a \emph{balanced wall crossing} (See Section \ref{sect:mutations}).

Under a hypothesis on $\omega_X$, a balanced wall crossing gives rise to an equivalance $\psi_w : \D^b(X^{ss}_-/G) \to \D^b(X^{ss}_+/G)$ which depends on a choice of $w\in \bZ$, and the composition $\Phi_w := \psi_{w+1}^{-1} \psi_w$ defines an autoequivalence of $\D^b(X^{ss}_-/G)$. Autoequivalences of this kind have been studied recently under the name window-shifts \cite{DS,Se}. We generalize the observations of those papers in showing that $\Phi_w$ is always a spherical twist.

Recall that if $B$ is an object in a dg-category, then we can define the twist functor
$$T_B : F \mapsto \op{Cone} (\op{Hom}^\bdot(B,F) \otimes_\bC B \to F)$$
If $B$ is a spherical object, then $T_B$ is by definition the spherical twist autoequivalence defined by $B$. More generally, if $S : \cA \to \cB$ is a spherical functor (Definition \ref{def:spherical_functor}), then one can define a twist autoequivalence $T_S := \op{Cone} (S\circ S^R \to \id_\cB)$ of $\cB$, where $S^R$ denotes the right adjoint. Throughout this paper we refer to a twist autoequivalence corresponding to a spherical functor simply as a ``spherical twist." A spherical object corresponds to the case where $\cA = \D^b(k-\op{vect})$.

It was noticed immediately \cite{ST01} that if $B$ were instead an exceptional object, then $T_B$ is the formula for the left mutation equivalence ${}^\perp B \to B^\perp$ coming from a pair of semiorthogonal decompositions $\sod{B^\perp, B} = \sod{B, {}^\perp B}$.\footnote{Such semiorthogonal decompositions exist when $\op{Hom}^\bdot(F^\bdot, B)$ has finite dimensional cohomology for all $F^\bdot$.} In fact, we will show that there is more than a formal relationship between spherical functors and mutations. If $\cC$ is a pre-triangulated dg category, then the braid group on $n$-strands acts by left and right mutation on the set of length $n$ semiorthogonal decompositions $\cC = \sod{\cA_1,\ldots,\cA_n}$ with each $\cA_i$ admissible. Mutating by a braid gives equivalences $\cA_i \to \cA^\prime_{\sigma(i)}$, where $\sigma$ is the permutation that the braid induces on end points. In particular if one of the semiorthogonal factors is the same subcategory before and after the mutation, one gets an autoequivalence $\cA_i \to \cA_i$.

\begin{sumthm}[spherical twist=mutation=window shifts] \label{sumthm:spherical_twists}

If $\cC$ is a pre-triangulated dg category admitting a semiorthogonal decomposition $\cC = \sod{\cA,\cG}$ which is fixed by the braid (acting by mutations)
\begin{center}
%\begin{tikzpicture}
%\braid[rotate=90, style strands={1}{black}, style strands={2}{black}] s_1 s_1 s_1 s_1;
%\end{tikzpicture}
\includegraphics[scale=.3]{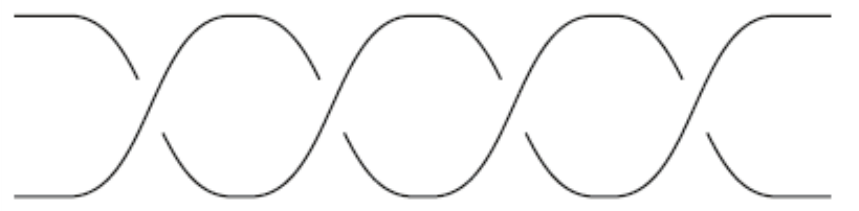}
\end{center}
then the autoequivalence of $\cG$ induced by mutation is the twist $T_S$ corresponding to a spherical functor $S : \cA \to \cG$ (Theorem \ref{thm:mut_are_sph}). Conversely, if $S : \cA \to \cB$ is a spherical functor, then there is a larger category $\cC$ admitting a semiorthogonal decomposition fixed by this braid which recovers $S$ and $T_S$ (Theorem \ref{thm:sph_fun_are_mutations}).

In the context of a balanced GIT wall crossing, the category $\cC$ arises naturally as a subcategory of the equivariant category $\D^b(X/G)$, defined in terms of ``grade restriction rules" (Section \ref{sect:background}). The resulting autoequivalence agrees with the window shift $\Phi_w$ (Proposition \ref{prop:window_shift_formula}) and corresponds to a spherical functor $f_w : \D^b(Z/L)_w \to \D^b(X^{ss}_-/G)$, where $Z/L$ is the ``critical locus" of the VGIT, which is unstable in both quotients (Section \ref{sect:mutations}).
\end{sumthm}

In the second half of the paper we revisit the prediction of derived autoequivalences from mirror symmetry. Spherical twist autoequivalences of $\D^b(V)$ for a Calabi-Yau $V$ correspond to loops in the moduli space of complex structures on the mirror Calabi-Yau $V^\dual$, and flops correspond, under the mirror map, to certain paths in that complex moduli space. We review these predictions, first studied in \cite{Ho05} for toric varieties, and formulate corresponding predictions for flops coming from VGIT in which an explicit mirror may not be known.

By studying toric flops between toric Calabi-Yau varieties of Picard rank $2$ (Section \ref{sect:toric}), we find that mirror symmetry predicts more autoequivalences than constructed in Theorem \ref{sumthm:spherical_twists}. The expected number of autoequivalences agrees with the length of a full exceptional collection on the critical locus $Z/L$ of the VGIT. Motivated by this observation, we introduce a notion of ``fractional grade restriction windows" given the data of a semiorthogonal decomposition on the critical locus. This leads to

\begin{sumthm}[Factoring spherical twists]
Given a full exceptional collection $\D^b(Z/L)_w = \langle E_0,\ldots,E_n \rangle$, the objects $S_i := f_w(E_i) \in \D^b(X^{ss}_-/G)$ are spherical, and (Corollary \ref{cor:factoring_twists_FEC})
$$\Phi_w = T_{S_0} \circ \cdots T_{S_n}.$$
This is a general phenomenon as well. Let $S = \cE \to \cG$ be a spherical functor of dg-categories and let $\cE = \sod{\cA,\cB}$ be a semiorthogonal decomposition such that there is also a semiorthogonal decomposition $\cE = \sod{F_S (\cB), \cA}$, where $F_S$ is the cotwist autoequivalence of $\cE$ induced by $S$. Then the restrictions $S_\cA : \cA \to \cG$ and $S_\cB : \cB \to \cG$ are spherical as well, and $T_S \simeq T_{S_\cA} \circ T_{S_\cB}$ (Theorem \ref{thm:factorize}).
\end{sumthm}

We propose an interpretation of this factorization theorem in terms of monodromy of the quantum connection in a neighborhood of a partial large volume limit (Section \ref{sect:fract_shifts}).

\subsection{Acknowledgements}

The first author would like to thank the attendees of the School on Algebraic Geometry and Theoretical Physics at the University of Warwick, July 2012 for many stimulating mathematical conversations, with special thanks to Will Donovan, Ed Segal, and Timothy Logvinenko for explaining their work. He also benefited greatly from conversations with Kentaro Hori, Paul Horja, Hiroshi Iritani, Lev Borisov, and Alexander Kuznetsov. The second author was partially supported by National Science Foundation award DMS-1204733. The first author was partially supported by National Science Foundation award DMS-1303960, by Columbia University, and by the Institute for Advanced Study.

%%%%%		%%%%%		%%%%%	
%%%%%		%%%%%		%%%%%
%%%%%		%%%%%		%%%%%
%%%%%		%%%%%		%%%%%

\section{Derived Kirwan surjectivity} \label{sect:background}

In this section we fix our notation and recall the theory of derived Kirwan surjectivity developed in \cite{HL12}. We also introduce the category $\cC_w$ and its semiorthogonal decompositions, which will be used throughout this paper.

We consider a smooth projective-over-affine variety $X$ over an algebraically closed field $k$ of characteristic $0$, and we consider a reductive group $G$ acting on $X$. Given a $G$-ample equivariant line bundle $L$, geometric invariant theory defines an open semistable locus $X^{ss} \subset X$. After choosing an invariant inner product on the cocharacter lattice of $G$, the Hilbert-Mumford numerical criterion produces a special stratification of the unstable locus by locally closed $G$-equivariant subvarieties $X^{us} = \bigcup_i S_i$ called Kirwan-Ness (KN) strata. The indices are ordered so that the closure of $S_i$ lies in $\bigcup_{j \geq i} S_j$.

Each stratum comes with a distinguished one-parameter subgroup $\lambda_i : \bC^\ast \to G$ and $S_i$ fits into the diagram
\begin{equation} \label{eqn:main_stratum_diagram}
\xymatrix{Z_i \ar@/^/[r]^-{\sigma_i} & Y_i \subset S_i := G \cdot Y_i \ar@/^/[l]^-{\pi_i}  \ar[r]^(.7){j_i} & X},
\end{equation}
where $Z_i$ is an open subvariety of $X^{\lambda_i \text{ fixed}}$, and $$Y_i = \left\{ \left. x \in X - \bigcup_{j>i} S_j \right| \lim_{t \to 0} \lambda_i(t) \cdot x \in Z_i \right\}.$$
$\sigma_i$ and $j_i$ are the inclusions and $\pi_i$ is taking the limit under the flow of $\lambda_i$ as $t\to 0$. We denote the immersion $Z_i \to X$ by $\sigma_i$ as well. Throughout this paper, the spaces $Z,Y,S$ and morphisms $\sigma,\pi,j$ will refer to diagram \ref{eqn:main_stratum_diagram}.

In addition, $\lambda_i$ determines the parabolic subgroup $P_i$ of elements of $G$ which have a limit under conjugation by $\lambda_i$, and the centralizer of $\lambda_i$, $L_i \subset P_i \subset G$, is a Levi component for $P_i$. One key property of the KN stratum is that $S_i = G \times_{P_i} Y_i$, so that $G$ equivariant quasicoherent sheaves on $S_i$ are equivalent to $P_i$-equivariant quasicoherent sheaves on $Y_i$. When $G$ is abelian, then $G=P_i=L_i$, and $Y_i = S_i$ is already $G$ invariant, so the story simplifies quite a bit.

\begin{thm}[derived Kirwan surjectivity, \cite{HL12}] \label{thm:derived_Kirwan_surjectivity}
Let $\eta_i$ be the weight of $\det(N^\dual_{S_i} X)|_{Z_i}$ with respect to $\lambda_i$. Choose an integer $w_i$ for each stratum and define the full subcategory
$$\G_w := \{ F^\bdot \in \D^b(X/G) | \forall i,  \sigma_i^\ast F^\bdot \text{ has weights in } [w_i,w_i+\eta_i) \text{ w.r.t. }\lambda_i \}.$$
Then the restriction functor $r : \G_w \to \D^b(X^{ss}/G)$ is an equivalence of dg-categories.
\end{thm}
The weight condition on $\sigma_i^\ast F^\bdot$ is called the \emph{grade restriction rule} and the interval $[w_i,w_i+\eta_i)$ is the \emph{grade restriction window}.  The theorem follows immediately from the corresponding statement for a single closed KN stratum by considering the chain of open subsets $X^{ss} \subset X_n \subset \cdots \subset X_0 \subset X$ where $X_i = X_{i-1} \setminus S_i$.

The full version of the theorem also describes the kernel of the restriction functor $r : \D^b(X/G) \to \D^b(X^{ss}/G)$. For a single stratum $S$ we define the full subcategory
$$\cA_w := \left\{ F^\bdot \in \D^b(X/G) \left| \begin{array}{c} \cH^\ast (\sigma^\ast F^\bdot) \text{ has weights in }[w,w+\eta] \text{ w.r.t. } \lambda \\ \cH^\ast (F^\bdot) \text{ supported on } S \end{array} \right. \right\}$$
we have an infinite semiorthogonal decomposition
$$\D^b(X/G) = \langle \ldots, \cA_{w-1}, \cA_w, \G_w , \cA_{w+1}, \ldots \rangle$$
This means that the subcategories are disjoint, semiorthogonal (there are no $R\op{Hom}$'s pointing to the left), and that every object has a functorial filtration whose associated graded pieces lie in these subcategories (ordered from right to left).\footnote{The early definitions of semiorthogonal decompositions required the left and right factors to be admissible, but this requirement is not relevant to our analysis. The notion we use is sometimes referred to as a \emph{weak semiorthogonal decomposition}} These categories are not obviously disjoint, but it is a consequence of the theory that no non-zero object supported on $S$ can satisfy the grade restriction rule defining $\G_w$.

Let $\D^b(Z/L)_w \subset \D^b(Z/L)$ denote the full subcategory which has weight $w$ with respect to $\lambda$, and let $(\bullet)_w$ be the exact functor taking the summand with $\lambda$ weight $w$ of a coherent sheaf on $Z/L$.
\begin{lem}[\cite{HL12}] \label{lem:unstable_category}
The functor $\iota_w : \D^b(Z/L)_w \to \cA_w$ is an equivalence, and its inverse can be described either as $(\sigma^\ast F^\bdot)_w$ or as $(\sigma^\ast F^\bdot)_{w+\eta} \otimes \det(N_S X)$.
\end{lem}
Using the equivalences $\iota_w$ and $r$ we can rewrite the main semiorthogonal decomposition
\begin{equation} \label{eqn:main_SOD}
\D^b(X/G) = \langle \ldots, \D^b(Z/L)_w, \D^b(X^{ss}/G)_w , \D^b(Z/L)_{w+1}, \ldots \rangle
\end{equation}
When there are multiple strata, one can inductively construct a nested semiorthogonal decomposition using $\D^b(X_{i-1} / G) = \langle \ldots, \cA^i_w, \D^b(X_i/G), \cA^i_{w+1}, \ldots \rangle$.

In this paper, we will consider the full subcategory
$$\cC_w := \{ F^\bdot \in \D^b(X/G) \left| \cH^\ast (\sigma^\ast F^\bdot) \text{ has weights in } [w,w+\eta] \text{ w.r.t. } \lambda \right \} \subset \D^b(X/G)$$
If we instead use the grade restriction window $[w,w+\eta)$, then we get the subcategory $\G_w \subset \cC_w$. The main theorem of \cite{HL12} implies that we have two semiorthogonal decompositions
\begin{equation} \label{eqn:two_semi_decomps}
\cC_w = \langle \G_w, \cA_w \rangle = \langle \cA_w , \G_{w+1} \rangle.
\end{equation}
We regard restriction to $X^{ss}$ as a functor $r : \cC_w \to \D^b(X^{ss}/G)$. The subcategory $\cA_w$ is the kernel of $r$, but is described more explicitly as the essential image of the fully faithful functor $\iota_w : \D^b(Z/L)_w \to \cC_w$ as discussed above.
\begin{lem} \label{lem:adjoint_unstable}
The left and right adjoints of $\iota_w : \D^b(Z/L)_w \to \cC_w$ are $\iota_w^L(F^\bdot) = (\sigma^\ast F^\bdot)_w$ and $\iota_w^R(F^\bdot) = (\sigma^\ast F^\bdot)_{w+\eta} \otimes \det N_S X|_Z$.
\end{lem}
\begin{proof}
Letting $G^\bdot \in \D^b(Z/L)_w$, we have $\op{Hom}_{X/G}(F^\bdot, \iota_w G^\bdot) \simeq \op{Hom}_{S/G}(j^\ast F^\bdot, \pi^\ast F^\bdot)$ and $\pi^\ast G^\bdot \in \D^b(S/G)_w$. In \cite{HL12}, we show that $\D^b(S/G)$ admits a baric decomposition, and using the baric truncation functors
\begin{align*}
\op{Hom}_{S/G}(j^\ast F^\bdot, \pi^\ast G^\bdot) &\simeq \op{Hom}_{S/G}(\ladj{w+1}j^\ast F^\bdot, \pi^\ast G^\bdot) \\
&\simeq \op{Hom}_{Z/L}(\sigma^\ast \ladj{w+1} j^\ast F^\bdot, G^\bdot)
\end{align*}
Where the last equality uses the fact that $\pi^\ast : \D^b(Z/L)_w \to \D^b(S/G)_w$ is an equivalence with inverse $\sigma^\ast$. Finally, we have $\sigma^\ast \ladj{w+1} j^\ast F^\bdot = (\sigma^\ast j^\ast F^\bdot)_w = (\sigma^\ast F^\bdot)_w$.

The argument for $\iota^R$ is analogous, but it starts with the adjunction for $j^! F^\bdot \simeq j^!(\cO_X) \otimes j^\ast F^\bdot$, $\op{Hom}_{X/G}(\iota_w G^\bdot, F^\bdot) \simeq \op{Hom}_{S/G}(\pi^\ast G^\bdot, \det (N_{S} X) \otimes j^\ast F^\bdot)$.
\end{proof}

\begin{lem} \label{lem:adjoints_restriction}
The functor $r : \cC_w \to \D^b(X^{ss}/G)$ has right and left adjoints given respectively by $r^R : \D^b(X^{ss}/G) \simeq \G_w \subset \cC_w$ and $r^L: \D^b(X^{ss}/G) \simeq \G_{w+1} \subset \cC_{w+1}$.
\end{lem}

Now because we have two semiorthogonal decompositions in Equation \eqref{eqn:two_semi_decomps}, there is a left mutation \cite{Bo89} equivalence functor $\lmut_{\cA_w} : \G_{w+1} \to \G_w$ defined by the functorial exact triangle
\begin{equation}\label{eqn:define_l_mutate}
\iota_w \iota_w^R (F^\bdot) \to F^\bdot \to \lmut_{\cA_w} F^\bdot \parr
\end{equation}
Note that restricting to $X^{ss}/G$, this triangle gives an equivalence $r (F^\bdot) \simeq r (\lmut_{\cA_w} F^\bdot)$. Thus this mutation implements the 'window shift' functor
\begin{equation}\label{eqn:mutation_is_window_shift}
\xymatrix{ \G_{w+1} \ar[dr]_r \ar[rr]^{\lmut_{\cA_w}} & & \G_w \\ & \D^b(X^{ss}/G) \ar[ur]_{r^{-1} = r^R} }
\end{equation}
meaning that $\lmut_{\cA_w} F^\bdot$ is the unique object of $\G_w$ restricting to the same object as $F^\bdot$ in $\D^b(X^{ss}/G)$.

%%%%%		%%%%%		%%%%%
%%%%%		%%%%%		%%%%%
%%%%%		%%%%%		%%%%%
%%%%%		%%%%%		%%%%%

\subsection{The category $\D^b(Z/L)_w$} \label{sect:describe_category_w}

We will provide a more geometric description of the subcategory $\D^b(Z/L)_w$. We define the quotient group $L^\prime = L / \lambda(\bC^\ast)$. Because $\lambda(\bC^\ast)$ acts trivially on $Z$, the group $L^\prime$ acts naturally on $Z$ as well.

\begin{lem}
The pullback functor gives an equivalence $\D^b(Z/L^\prime) \xrightarrow{\simeq} \D^b(Z/L)_0$.
\end{lem}
\begin{proof} This follows from the analogous statement for quasicoherent sheaves, which is a consequence of descent.\end{proof}

The categories $\D^b(Z/L)_w$ can also be related to $\D^b(Z/L^\prime)$. If $\lambda  : \bC^\ast \to G$ has the kernel $\mu_n \subset \bC^\ast$, then $\D^b(Z/L)_w = \emptyset$ unless $w \equiv 0 \mod n$. In this case we replace $\lambda$ with an injective $\lambda^\prime$ such that $\lambda = (\lambda^\prime)^n$ and $\D^b(Z/L)_{[\lambda^\prime = w]} = \D^b(Z/L)_{[\lambda = nw]}$. Thus we will assume that $\lambda$ is injective.

\begin{lem} \label{lem:trivialize_gerbe}
Let $\cL \in \D^b(Z/L)_w$ be an invertible sheaf. Then pullback followed by $\cL \otimes \bullet$ gives an equivalence $\D^b(Z/L^\prime) \xrightarrow{\simeq} \D^b(Z/L)_w$.
\end{lem}

For instance, if there is a character $\chi : L \to \bC^\ast$ such that $\chi \circ \lambda$ is the identity on $\bC^\ast$, then $\chi$ induces an invertible sheaf on $Z/L$ with weight $1$, so Lemma \ref{lem:trivialize_gerbe} applies. If $G$ is abelian then such a character always exists.

\begin{rem}
This criterion is not always met, for example when $Z/L = \pt / GL_n$ and $\lambda$ is the central $\bC^\ast$. What is true in general is that $Z/L \to Z/L^\prime$ is a $\bC^\ast$ gerbe, and the category $\D^b(Z/L)_1$ is by definition the derived category of coherent sheaves on $Z/L^\prime$ twisted by that gerbe. The data of an invertible sheaf $\cL \in \D^b(Z/L)_1$ is equivalent to a trivialization of this gerbe $Z/L \xrightarrow{\simeq} Z/L^\prime \times \pt / \bC^\ast$.
\end{rem}

%%%%%		%%%%%		%%%%%
%%%%%		%%%%%		%%%%%
%%%%%		%%%%%		%%%%%
%%%%%		%%%%%		%%%%%

\section{Window shift autoequivalences, mutations, and spherical functors}
\label{sect:mutations}

In this paper we study balanced GIT wall crossings. Let $L_0$ be a $G$-ample line bundle such that the strictly semistable locus $X^{sss} = X^{ss} - X^s$ is nonempty, and let $L^\prime$ be another $G$-equivariant line bundle. We assume that $X^{ss}=X^s$ for the linearizations $L_\pm = L_0 \pm \epsilon L^\prime$ for sufficiently small $\epsilon$, and we denote $X^{ss}_\pm = X^{ss}(L_\pm)$. In this case, $X^{ss}(L_0) - X^{ss}(L_\pm)$ is a union of KN strata for the linearization $L_\pm$, and we will say that the wall crossing is \emph{balanced} if the strata $S^+_i$ and $S^-_i$ lying in $X^{ss}(L_0)$ are indexed by the same set, with $Z^+_i = Z^-_i$ and $\lambda_i^+ = (\lambda_i^-)^{-1}$. This is slightly more general than the notion of a truly faithful wall crossing in \cite{DH}. In particular, if $G$ is abelian and there is some linearization with a stable point, then all codimension one wall crossings are balanced.

In this case we will replace $X$ with $X^{ss}(L_0)$ so that these are the only strata we need to consider. In fact we will mostly consider a balanced wall crossing where only a single stratum flips -- the analysis for multiple strata is analogous. We will drop the superscript from $Z^\pm$, but retain superscripts for the distinct subcategories $\cA^\pm_w$. Objects in $\cA^\pm_w$ are supported on $S^\pm$, which are distinct because $S^+$ consist of orbits of points flowing to $Z$ under $\lambda^+$, whereas $S^-$ consists of orbits of points flowing to $Z$ under $\lambda^-$. When there is ambiguity as to which $\lambda^\pm$ we are referring to, we will include it in the notation, i.e. $\D^b(Z/L)_{[\lambda^+ = w]}$.

\begin{observation} \label{observation:wall_crossing}
If $\omega_X|_Z$ has weight $0$ with respect to $\lambda^\pm$, then $\eta^+ = \eta^-$ (see \cite{HL12}). This implies that $\cC^+_w = \cC^-_{w^\prime}$, $\G^+_w = \G^-_{w^\prime+1}$, and $\G^+_{w+1} = \G^-_{w^\prime}$, where $w^\prime = -\eta - w$.
\end{observation}

This observation, combined with derived Kirwan surjectivity, implies that the restriction functors $r_\pm : \G_w^- \to \D^b(X^{ss}_\pm / G)$ are both equivalences. In particular $\psi_w := r_+ r_-^{-1} : \D^b(X^{ss}_-/G) \to \D^b(X^{ss}_+/G)$ is a derived equivalence between the two GIT quotients. Due to the dependence on the choice of $w$, we can define the \emph{window shift autoequivalence} $\Phi_w := \psi_{w+1}^{-1} \psi_w$ of $\D^b(X^{ss}_-/G)$.

\begin{lem}
If there is an invertible sheaf $\cL \in \D^b(X/G)$ such that $\cL|_Z$ has weight $w$ w.r.t. $\lambda^+$, then $\Phi_w = (\cL^\dual \otimes) \Phi_0 (\cL \otimes)$. In particular, if $\cL$ has weight $1$, then for any $v,w$, $\psi_v^{-1} \psi_w$ lies in the subgroup of $\op{Aut} \D^b(X^{ss}_- / G)$ generated by $\Phi_0$ and $\cL \otimes$.
\end{lem}
\begin{proof} The commutativity of the following diagram implies that $(\cL^\dual \otimes) \psi_k (\cL\otimes) = \psi_{k+w}$
$$\xymatrix@R=8pt{\D^b(X^{ss}_-/G) \ar[d]^{\otimes \cL} & \G^-_{k+w} \ar[r] \ar[l] \ar[d]^{\otimes \cL} & \D^b(X^{ss}_+/G) \ar[d]^{\otimes \cL} \\ \D^b(X^{ss}_-/G) & \G^-_{k} \ar[r] \ar[l] & \D^b(X^{ss}_+/G) }$$
\end{proof}

Here we are able to give a fairly explicit description of $\Phi_w$ from the perspective of mutation. Note that because $\G^+_{w+1} = \G^-_{w^\prime}$, the inverse of the restriction $\G^+_{w+1} \to \D^b(X^{ss}_- / G)$ is the right adjoint $r_-^R$, whereas the inverse of the restriction $\G^-_{w^\prime+1} \to \D^b(X^{ss}_- / G)$ was the left adjoint $r_-^L$ by Lemma \ref{lem:adjoints_restriction}.

\begin{prop}\label{prop:wind_mut}
The autoequivalence $\Phi_w$ of $\D^b(X^{ss}_- / G)$ makes the following diagram commute, i.e. $\Phi_w = r_- \circ \lmut_{\cA^+_w} \circ r_-^R$.
$$
\xymatrix{
\G^+_{w+1} \ar[r]^{\lmut_{\cA^+_{w}}} & \G^+_{w} \ar[d]^{r_-} \\ 
\ar[u]^{r^R_- = r^{-1}_-} \ar[r]^{\Phi_w} \D^b(X^{ss}_- / G) & \D^b(X^{ss}_- / G) 
}
$$
\end{prop}
\begin{proof} This is essentially rewriting Diagram \eqref{eqn:mutation_is_window_shift} using Observation \ref{observation:wall_crossing} and its consequences.
\end{proof}

%%%%%		%%%%%		%%%%%
%%%%%		%%%%%		%%%%%
%%%%%		%%%%%		%%%%%
%%%%%		%%%%%		%%%%%

\subsection{Window shifts are spherical twists}\label{subsect:win_shifts_are_spherical}
Next we show that the window shift autoequivalence $\Phi_w$ is a twist corresponding to a spherical functor.

This generalizes a spherical object \cite{ST01}, which is equivalent to a spherical functor $\D^b(k-vect) \to \cB$. By describing window shifts both in terms of mutations and as spherical twists, we show why these two operations have the ``same formula" in this setting. In fact, in the next section we show that spherical twists can always be described by mutations.

Let $E: = Y^+ \cap X^{ss}_-$, it is $P_+$-equivariant, and let $\tilde{E} = S^+ \cap X^{ss}_- = G \cdot E$. Then we consider the diagram
\begin{equation} \label{eqn:EZ_diagram}
\xymatrix{ E / P_+ = \tilde{E} / G \ar[d]^\pi \ar[r]^j & X_-^{ss}/G \\ Z/L & }
\end{equation}
This is a stacky form of the EZ-diagram used to construct autoequivalences in \cite{Ho05}. We define the transgression along this diagram $f_w = j_\ast \pi^\ast : \D^b(Z/L)_w \to \D^b(X^{ss}_- /G)$. Note that we have used the same letters $\pi$ and $j$ for the restriction of these maps to the open substack $E/P_+ \subset Y^+/P_+$, but we denote this transgression $f_w$ to avoid confusion.

\begin{prop}\label{prop:window_shift_formula}
The window shift functor $\Phi_w$ is defined for $F^\bdot \in \D^b(X^{ss}_- / G)$ by the functorial mapping cone
$$f_w f_w^R (F^\bdot) \to F^\bdot \to \Phi_w(F^\bdot) \parr$$
\end{prop}

\begin{proof}
This essentially follows from abstract nonsense. By the definition of left mutation \eqref{eqn:define_l_mutate}, and by the fact that $r_- r^R_- = \id_{\D^b(X_-^{ss}/G)}$, it follows that the window shift autoequivalence is defined by the cone
$$r_- \iota^+_w (\iota^+_w)^R r_-^R (F^\bdot) \to F^\bdot \to \Phi_w (F^\bdot) \parr$$
Furthermore, by construction we have $f_w = r_- \iota_w^+$, so $f_w^R \simeq (\iota_w^+)^R r_-^R$. The claim follows.
\end{proof}

Consider the case where $\D^b(Z/L)_w$ is generated by a single exceptional object $E$. The object $E^+ := \iota_w^+ E \in \cA^+_w$ is exceptional, and the left mutation functor \eqref{eqn:define_l_mutate} acts on $F^\bdot \in \G^+_{w+1}$ by
$$\op{Hom}_{X/G} (E^+, F^\bdot) \otimes E^+  \to F^\bdot \to \lmut_{\cA^+_w} (F^\bdot) \parr$$
To emphasize the dependence on $E^+$ we write $\lmut_{E^+} := \lmut_{\cA^+_w}$.  As we have shown, $\lmut_{E^+}(F^\bdot)|_{X^{ss}_-}$ is the window shift autoequivalence $\Phi_w(F^\bdot|_{X^{ss}_-})$. If we restrict the defining exact triangle for $\lmut_{E^+}(F^\bdot)$ to $X^{ss}_-$ we get
$$\op{Hom}_{X/G} (E^+,F^\bdot) \otimes E^+|_{X^{ss}_-} \to F^\bdot|_{X^{ss}_-} \to \Phi_w(F^\bdot|_{X^{ss}_-}) \parr$$
Define the object $S = E^+|_{X^{ss}_-} \in \D^b(X^{ss}_-/G)$. The content of Proposition \ref{prop:window_shift_formula} is that the canonical map $\op{Hom}_{X/G} (E^+,F^\bdot) \to \op{Hom}_{X^{ss}_-/G}(S , F^\bdot|_{X^{ss}_-} )$ is an isomorphism, so that $\Phi_w = \lmut_{E^+}|_{X^{ss}_-}$ is the spherical twist $T_S$ by the object $S$. This can be verified more directly using the following
\begin{lem} \label{lem:local_cohomology}
For $F^\bdot, G^\bdot \in \cC^-_{w^\prime}$,
$$R\Gamma_{S^-} \inner{\op{Hom}}_{X/G}(F^\bdot,G^\bdot) \simeq \op{Hom}_{Z/L} ((\sigma^\ast F^\bdot)_{w^\prime}, (\sigma^\ast G^\bdot \otimes \kappa_-)_{w^\prime} ).$$
Equivalently we have an exact triangle
$$\op{Hom}_{Z/L} ((\sigma^\ast F^\bdot)_{w^\prime}, (\sigma^\ast G^\bdot \otimes \kappa_-)_{w^\prime} ) \to \op{Hom}_{X/G}(F^\bdot, G^\bdot) \to \op{Hom}_{X^{ss}_- / G} (F^\bdot|_{X^{ss}_-},G^\bdot|_{X^{ss}_-}) \parr$$
\end{lem}
\begin{proof}
Let $\cC = \sod{\cA,\cB}$ be a semiorthogonal decomposition of a pretriangulated dg-category, and let $\iota_\cA$ and $\iota_\cB$ be the inclusions. Applying $\op{Hom}(F, \bullet)$ to the canonical exact triangle $\iota_\cB \iota_\cB^R G \to G \to \iota_\cA \iota_\cA^L G \parr$ gives the exact triangle
$$\op{Hom}_\cB (\iota_\cB^L F, \iota_\cB^R G) \to \op{Hom}_\cC (F,G) \to \op{Hom}_\cA (\iota_\cA^L F, \iota_\cA^L G) \parr$$
assuming $\cB$ is left admissible. The lemma is just a special case of this fact for the semiorthogonal decomposition $\cC^-_{w^\prime} = \sod{\G^-_{w^\prime}, \cA^-_w}$, using the description of the adjoint functors in Lemma \ref{lem:adjoint_unstable}.
\end{proof}

In summary, we have given a geometric explanation for the identical formulas for $L_{E^+}$ and $T_S$: the spherical twist is the restriction to the GIT quotient of a left mutation in the equivariant derived category.

\begin{ex}
Let $X$ be the crepant resolution of the $A_n$ singularity. It is the $2$ dimensional toric variety whose fan in $\bZ^2$ has rays spanned by $(1,i)$, for $i=0,\ldots,n+1$, and which has a $2$-cone for each pair of adjacent rays. Removing one of the interior rays corresponds to blowing down a rational curve $\bP^1 \subset X$ to an orbifold point with $\bZ/2\bZ$ stabilizer. This birational transformation can be described by a VGIT in which $Z / L^\prime \simeq \pt$. The spherical objects corresponding to the window shift autoequivalences are $\cO_{\bP^1}(m)$.
\end{ex}

\begin{rem}
Horja \cite{Ho05} introduced the notion of an EZ-spherical object $F^\bdot \in \D^b(E/P_+)$ for a diagram $Z/L^\prime \xleftarrow{q} E/P_+ \xrightarrow{j} X^{ss}_- / G$ -- his notion is equivalent to the functor $j_\ast (F^\bdot \otimes q^\ast (\bullet))$ being spherical \ref{def:spherical_functor}. Proposition \ref{prop:window_shift_formula} amounts to the fact that $\cO_{E/P_+}$ is an EZ-spherical object for this diagram. By the projection formula $q^\ast \cL$ is EZ-spherical for any invertible sheaf $\cL$ on $Z$. The twist functors corresponding to different choices of $\cL$ are equivalent.
\end{rem}

\begin{rem}
Our results also extend results in \cite{Se,DS}.  The first work formally introduced grade restriction windows to the mathematics literature and showed that window shift equivalences are given by spherical functors in the context of gauged Landau-Ginzburg models.  (See subsection \ref{subsect:complete_intersections}.)  In the second work, the authors study window shift autoequivalences associated to Grassmannian flops, using representation theory of $GL(n)$ to compute with homogeneous bundles. 
\end{rem}

%%%%%		%%%%%		%%%%%
%%%%%		%%%%%		%%%%%
%%%%%		%%%%%		%%%%%
%%%%%		%%%%%		%%%%%

\subsection{All spherical twists are mutations}
\label{sect:sphmut}

We have shown that the window shift $\Phi_w$ is a twist $\op{Cone}(f_w f_w^R \to \id)$ corresponding to a functor $f_w : \D^b(Z/L)_w \to \D^b(X^{ss}_- / G)$. Now we show that this $f_w$ is \emph{spherical} \cite{A}, and in fact any autoequivalence of a dg-category arising from mutations as $\Phi_w$ does is a twist by a spherical functor. Conversely, any spherical functor between dg-categories with a compact generator arises from mutations.

Using the equalities of Observation \ref{observation:wall_crossing}, we have the following semiorthogonal decompositions of $\cC_w^+ = \cC^-_{w^\prime}$, all coming from \eqref{eqn:two_semi_decomps}:
\begin{equation} \label{eqn:wall_cross_semidecomp}
\xymatrix{\sod{\cA^+_w, \G^+_{w+1} } \ar@{=>}[r]^{L_{\cA^+_w}} \ar@{<=} `d[r] `[rrr]_{L_{\G_{w+1}^+}} [rrr] & \sod{ \G^+_w, \cA^+_w } \ar@{=>}[r]^{L_{\G^+_w}} & \sod{ \cA^-_{w^\prime}, \G^+_w} \ar@{=>}[r]^{L_{\cA^-_{w^\prime}}} & \sod{ \G^+_{w+1}, \cA^-_{w^\prime}} }
\end{equation}
where we conclude a fortiori that each semiorthogonal decomposition arises from the previous one by left mutation. Each mutation gives an equivalence between the corresponding factors in each semiorthogonal decomposition, and the autoequivalence $\Phi_w$, interpreted as an autoequivalence of $\G^+_w$, is obtained by following the sequence of mutations.

\begin{rem}\label{rem:braid_functor}
The braid group $B_n$ on $n$ strands acts by mutations on the set of semiorthogonal decompositions of length $n$ (with admissible factors). The fact that the first and last semiorthogonal decompositions in \ref{eqn:wall_cross_semidecomp} are equal means that this semiorthogonal decomposition has a nontrivial stabilizer in $B_2$ under its action on length two semiorthogonal decompositions of $\cC_w$. We would like to point out that this may be a way to produce interesting autoequivalences more generally. Let $\cG_n$ be the groupoid whose objects are strong semiorthogonal decompositions (i.e. all factors are admissible subcategories) of length $n$ and whose morphisms are braids that take one to another by mutation. Let $e=\sod{\cA_1,\dotsc,\cA_n}$ be a semiorthogonal decomposition in the category of interest.  Then $\op{Aut}_{\cG_n}(e)$ is a subgroup of $B_n$ and for each $i$ there is a representation
$$ \op{Aut}(e) \to \op{Aut}(\cA_i), $$
the group of exact autoequivalences of $\cA_i$ up to isomorphism of functors. By construction the autoequivalences in the image of this representation are compositions of mutations.  In the situation above $B_2 = \bZ$ and $\op{Aut}( \sod{\G^+_w, \cA^+_w} ) \subset B_2$ is the index four subgroup.
\end{rem}

Let us recall the definition of spherical functor, which is motivated by \cite{A}.  However, one should see \cite{AL} for the definition below and a complete, rigorous treatment of spherical functors in the formalism of dg-categories and bimodules.

\begin{defn}[\cite{A,AL}]\label{def:spherical_functor}
A dg-functor $S:\cA \to \cB$ of pre-triangulated dg-categories is \emph{spherical} if it admits right and left adjoints $R$ and $L$ such that 
\begin{enumerate}
\item the cone $F_S$ of $\id \to R S$ is an autoequivalence of $\cA$, and
\item the natural morphism $R \to F_S L$ induced by $R \to R S L$ is an isomorphism of functors.
\end{enumerate}
If $S$ is spherical, the cone $T_S$ on the morphism $S R \to \id$ is an autoequivalence called the \emph{twist} corresponding to $S$.
\end{defn}

Suppose that $\cC$ is a pre-triangulated dg category admitting semiorthogonal decompositions
$$ \cC = \sod{\cA,\cB} = \sod{\cB,\cA'} = \sod{\cA',\cB'} = \sod{\cB',\cA}. $$
Denote by $i_\blt$ the inclusion functors.  Since $\cA,\cB,\cA',\cB'$ are admissible, $i_\blt$ admits right and left adjoints $i^R_\blt$ and $i^L_\blt$, respectively.  We can use these functors to describe the mutations
$$ \lmut_{\cA} = i^L_{\cB'}i_{\cB}: \cB \to \cB', \quad \rmut_{\cA} = i^R_{\cB} i_{\cB'}:\cB' \to \cB, $$
with analogous formulae for the other mutations.

\begin{thm}\label{thm:mut_are_sph}
The functor $S:\cA \to \cB$ given by $S = i^L_\cB i_\cA$ is spherical.  Moreover, the spherical twist $T_S:\cB \to \cB$ is obtained as the mutation 
$$ T_S \cong \lmut_{\cA'} \circ \lmut_{\cA}. $$
\end{thm}
\begin{proof}
We must produce left and right adjoints for $S$, then check the two parts of the definition.  Clearly the right adjoint to $S$ is $R = i_{\cA}^R i_\cB$.  In order to compute the left adjoint, we first apply $i^L_\cB$ to the triangle
$$  i_{\cB'}i^R_{\cB'} \to \id_{\cC} \to i_{\cA'} i^L_{\cA'} \parr. $$
Since $i^L_\cB i_{\cA'}=0$ we see that the map 
$$ \lmut_{\cA'} i^R_{\cB'} = i^L_{\cB} i_{\cB'} i^R_{\cB} \to i^L_{\cB} $$
is an isomorphism.  Using the fact that $\lmut_{\cA'}$ and $\rmut_{\cA'}$ are biadjoint, it follows that $L = i^L_\cA i_{\cB'} \rmut_{\cA'}$.

To establish (1), we will express $F_S$ in terms of mutations.  Begin with the triangle, 
\begin{equation}\label{eq:tri1}
 \id_{\cC} \to i_\cB i^L_\cB \to i_{\cA'}i^R_{\cA'}[1] \parr. 
 \end{equation}
Then apply $i^R_\cA$ on the right and $i_\cA$ on the left to get a triangle
$$ \id_{\cA} \to i^R_{\cA} i_\cB i^L_\cB i_\cA = RS \to i^R_\cA i_{\cA'}i^R_{\cA'}i_\cA[1] = \rmut_{\cB'} \rmut_{\cB}[1] \parr. $$
Since $i_\cA$ is fully faithful, the first map is the unit of the adjunction between $S$ and $R$ so we see that $F_S \cong \rmut_{\cB'}\rmut_{\cB}[1]$.  Hence it is an equivalence.  A very similar computation shows that $T_S \cong \lmut_{\cA'}\lmut_{\cA}$.

We now verify (2), that the composition $R \to RSL \to F_S L$ is an isomorphism.  The map $R \to F_S L$ is the composite
\begin{multline*} R = i^R_\cA i_\cB \to 
(i^R_\cA i_\cB)( i^L_\cB i_{\cB'} i^R_{\cB'} i_\cA)( i^L_\cA i_{\cB'} i^R_{\cB'} i_\cB) \to \\
 RSL = (i^R_\cA i_\cB)( i^L_\cB i_\cA)( i^L_\cA i_{\cB'} i^R_{\cB'} i_\cB) \to  
 F_S L = i^R_{\cA} i_{\cA'} i^R_{\cA'} i_{\cA} i^L_{\cA} i_{\cB'} i^R_{\cB'} i_\cB
\end{multline*}
where the middle map comes from the isomorphism $i^L_\cB i_{\cB'} i^R_{\cB'} i_\cA \to S = i^L_\cB i_\cA$ that we discussed in the preceding paragraphs.  The first map is obtained by applying $R \lmut_{\cA'}$ and $\rmut_{\cA'}$ to the left and right, respectively of the unit morphism $ \id_{\cB'} \to (i^R_{\cB'} i_\cA )(i^L_\cA i_{\cB'})$.  To get the last map one applies $i^R_\cA$ and $i_\cA i^L_\cA i_{\cB'} i^R_{\cB'} i_\cB$ to the left and right, respectively of the map $i_\cB i^L_\cB \to i_{\cA'} i^R_{\cA'}[1]$ from the triangle \eqref{eq:tri1}. 

In order to understand the morphism $R \to RSL$, consider the commutative diagram
$$
\xymatrix@=22pt{
i^L_\cB i_{\cB'} \ar[r]^(.4){\id \circ \epsilon} \ar[dr]_{=} & i^L_\cB i_{\cB'}  i^R_{\cB'} i_{\cB'} \ar[r]^(.45){\id \circ \epsilon \circ \id} \ar[d]^{\id \circ \eta \circ \id} & i^L_{\cB} i_{\cB'} i^R_{\cB'} i_\cA i^L_\cA i_{\cB'} \ar[d]^{\id \circ \eta \circ \id} \\
 & i^L_\cB i_{\cB'} \ar[r]^{\id \circ \epsilon \circ \id} & i^L_\cB i_{\cA} i^L_{\cA} i_{\cB'} }
$$
In this diagram, units and counits of adjunctions are denoted $\epsilon$ and $\eta$, respectively.  The map $R \to RSL$ is obtained by applying $i^R_\cA i_\cB$ and $i^R_{\cB'}i_{\cB}$ on the left and right, respectively, to the clockwise composition from the upper left to the lower right.  On the other hand the counterclockwise composite from the upper left to the lower right comes from the unit morphism $\id_\cC \to i_\cA i^L_\cA$ by applying $i^L_\cB$ and $i_{\cB'}$ on the right and left, respectively.  Therefore we get $R \to R S L$ by applying $i^R_\cA i_\cB i^L_\cB$ and $i_{\cB'} i^R_{\cB'}i_{\cB}$ to the left and right of this unit morphism, respectively.

Next, consider the commutative diagram
$$
\xymatrix@=22pt{
i_\cB i^L_\cB \ar[r]^{\id \circ \epsilon} \ar[d] & i_\cB i^L_\cB i_\cA i^L_\cA \ar[d] \\
i_{\cA'} i^R_{\cA'}[1] \ar[r]^{\id \circ \epsilon} & i_{\cA'} i^R_{\cA'} i_\cA i^L_\cA[1] 
}
$$
We have established now that the map $R \to F_S L$ is obtained from the clockwise composition in this diagram by applying $i^R_\cA$ and $i_{\cB'} i^R_{\cB'}i_\cB$ on the left and right, respectively.  Let us examine what happens when we apply these functors to the whole commutative diagram.  From the triangle \eqref{eq:tri1} and the fact that $i^R_\cA i_{\cB'} = 0$ we see that the left vertical map becomes an isomorphism.  Moreover, the unit map fits into the triangle
$$ i_\cB i^R_\cB \to \id_{\cC} \to i_\cA i^L_\cA \parr $$
and since $i^R_{\cA'} i_\cB = 0$ it follows that the bottom horizontal map becomes an isomorphism as well.  So $R \to F_S L$ is an isomorphism.
\end{proof}

\begin{rem}
There are other functors arising from the sequence of semiorthogonal decompositions in the statement of Theorem \ref{thm:mut_are_sph}, such as $i^R_{\cB^\prime} i_\cA : \cA \to \cB^\prime$, which are spherical because they are obtained from $S$ by composing with a suitable mutation. The corresponding spherical twist autoequivalences can also be described by mutation in $\cC$.
\end{rem}

We can also obtain a converse to this statement.  Suppose that $A^\bdot$ and $B^\bdot$ are dg-algebras over $k$.  Write $\D(\bullet)$ for the derived category of right dg modules over $\bullet$.  We begin with a folklore construction.  Let $F^\bdot$ be an $A^\bdot-B^\bdot$ bimodule defining a dg functor $F : D(A^\bdot) \to D(B^\bdot)$ given by $F(M^\bdot) = M^\bdot \otimes_A F^\bdot$.  Define a new dg algebra
$$ C_F = \begin{pmatrix} A^\bdot & F^\bdot \\ 0 & B^\bdot \end{pmatrix}. $$
More precisely, as a complex $C_F = A^\bdot \oplus F^\bdot \oplus B^\bdot$ and the multiplication is given by
$$ (a,f,b)(a',f',b') = (aa', af' + f b', bb').$$

By construction $C_F^\bdot$ has a pair of orthogonal idempotents $e_A=(1,0,0)$ and $e_B=(0,0,1)$. Every module splits as a complex $M^\bdot = M^\bdot_A \oplus M^\bdot_B$, where $M^\bdot_A := M^\bdot e_A$ is an $A^\bdot$ module and $M^\bdot_B := M^\bdot e_B$ is a $B^\bdot$ module. In fact the category of right $C_F^\bdot$ modules is equivalent to the category of triples consisting of $M^\bdot_A \in D(A^\bdot)$, $M^\bdot_B \in D(B^\bdot)$, and a structure homomorphism of $B^\bdot$ modules $M_A \otimes_A F^\bdot \to M^\bdot_B$, with intertwiners as morphisms. In order to abbreviate notation, we will denote the data of a module over $C^\bdot_F$ by its structure homomorphism $[F(M^\bdot_A) \to M^\bdot_B]$

Let $\cA,\cB \subset D(C^\bdot)$ be the full subcategories of modules of the form $[F(M^\bdot_A) \to 0]$ and $[F(0) \to M^\bdot_B]$ respectively. Then $\cA \simeq D(A^\bdot)$, $\cB \simeq D(B^\bdot)$, and the projection $D(C^\bdot_F) \to \cA$ (resp. $\cB$) given by $[F(M^\bdot_A) \to M^\bdot_B] \mapsto M^\bdot_A$ (resp. $M^\bdot_B$) is the left (resp. right) adjoint of the inclusion. We have semiorthogonality $\cB \perp \cA$ and a canonical short exact sequence
$$[F(0) \to M^\bdot_B] \to [F(M^\bdot_A) \to M^\bdot_B] \to [F(M^\bdot_A) \to 0] \parr $$
and therefore $\D(C_F^\bdot) = \sod{ \cA, \cB }$.

\begin{lem}\label{lemma:adj-funct-construction}
Suppose that $G^\bdot$ is a $B^\bdot-A^\bdot$ bimodule such that $\tensor G^\bdot$ is right adjoint to $\tensor F^\bdot$. Then there is an equivalence $\Phi:\D(C^\bdot_F) \to \D(C^\bdot_G)$ such that $\Phi$ restricts to the identity functor between the subcategories of $D(C^\bdot_F)$ and $D(C^\bdot_G)$ which are canonically identified with $D(A^\bdot)$.
\end{lem}
\begin{proof}
Note that the adjunction allows us to identify a module over $C^\bdot_F$ by a structure homomorphism $M^\bdot_A \to G(M^\bdot_B)$ rather than a homomorphism $F(M^\bdot_A) \to M^\bdot_B$. Letting $M = [M^\bdot_A \to G(M^\bdot_B)] \in D(C^\bdot_F)$, we define
$$\Phi(M)_A = \op{Cone}(M^\bdot_A \to G(M^\bdot_B))[-1] \qquad \Phi(M)_B = M_B[-1]$$
with the canonical structure homomorphism $G(M^\bdot_B) \to \Phi(M)_A$ defining an object in $D(C^\bdot_G)$. This construction is functorial.

For $N = [G(N_B) \to N_A] \in D(C^\bdot_G)$, the inverse functor assigns
$$\Phi^{-1}(N)_A = \op{Cone}(G(N^\bdot_B) \to N^\bdot_A) \qquad \Phi^{-1}(N)_B = N^\bdot_B[1]$$
with the canonical structure homomorphism $\Phi^{-1}(N)_A \to G(N^\bdot_B [1])$ defining an object of $D(C^\bdot_F)$ (again using the adjunction between $F$ and $G$).
\end{proof}

\begin{rem}
If $F^\bdot$ and $G^\bdot$ are perfect bimodules, then $\D(\bullet)$ can be replaced with $\Perf(\bullet)$ in the above lemma.
\end{rem}

Consider a functor $S:\D(A^\bdot) \to \D(B^\bdot)$, given by a $A^\blt-B^\blt$ bimodule $S^\bdot$, with right and left adjoints $R,L$ given by bimodules $R^\blt$ and $L^\blt$ respectively.  Fix morphisms $A^\blt \to S^\blt \tensor_{B^\blt} R^\blt$, etc., representing the units and co-units of the adjunctions.  Note that from these choices we can produce a bimodule $F_S^\blt$ representing $F_S$ and a quasi-isomorphism $R^\blt \to L^\blt \tensor_{A^\blt} F_S^\blt$.  

\begin{thm}\label{thm:sph_fun_are_mutations}
If $S$ is spherical then there is a pre-triangulated dg category $\cC$ which admits semiorthogonal decompositions
$$ \cC = \sod{\cA,\cB} = \sod{\cB,\cA'} = \sod{\cA',\cB'} = \sod{\cB',\cA}. $$
such that $S$ is the inclusion of $\cA$ followed by the projection onto $\cB$ with kernel $\cA'$, as in the previous theorem.
\end{thm}
\begin{proof}
Let $S' = S[-1]$, $R' = R[1]$, and $L' = L[1]$ and note that the same adjunctions hold.  We will show that $\cC = \D(C^\bdot_{S'})$ admits the desired semiorthogonal decomposition.  The reason for introducing $S',R',$ and $L'$ is the following.  We observe that $S \simeq i_\cB^L \circ i_\cA$. Indeed, $\cA'$ is the full subcategory of $\D(C^\bdot_{S[-1]})$ of objects $[ S[-1](M_A^\bdot) \to M_B^\bdot ]$ where the structural morphism is an isomorphism.  So we see that for any $M_A^\bdot$ there is a triangle
$$ [S[-1](M^\bdot_A) \to S[-1](M^\bdot_A)] \to [S[-1](M^\bdot_A) \to 0 ] \to [ S'(0) \to S(M^\bdot_A) ] \parr $$
Hence including $\cA$ and projecting to $\cB$ away from $\cA'$ gives $S$.

It follows from Lemma \ref{lemma:adj-funct-construction} that there are equivalences
$$ \xymatrix{ \D(C^\bdot_{R'}) \ar@/^/[r]^{\Psi_1} & \ar@/^/[l]^{\Phi_1} \D(C^\bdot_{S'}) \ar@/^/[r]^{\Psi_2} & \ar@/^/[l]^{\Phi_2} \D(C^\bdot_{L'}) } $$
By construction, $\D(C^\bdot_{S'})$ admits a semi-orthogonal decomposition $\sod{\cA,\cB}$.  We define two more full subcategories using the above equivalences.  Let $\cA' = \Phi_2 \D(A^\bdot)$ and $\cB' = \Psi_1 \D(B^\bdot)$.  Then we have the semiorthogonal decompositions
$$ \D(C^\bdot_{S'}) = \sod{ \cB',\cA } = \sod{ \cA, \cB } = \sod{ \cB,\cA' }. $$

All that remains is to show that we have a semiorthogonal decomposition $\D(C^\bdot_{S'}) = \sod{ \cA',\cB'}$ as well.  We will produce an autoequivalence of $\D(C^\bdot_{S'})$ which carries $\cA$ to $\cA'$ and $\cB$ to $\cB'$, establishing the existence of the remaining semiorthogonal decomposition.

The equivalence $F_S$ gives rise to another equivalence, $X:\D(C^\bdot_{L'}) \to \D(C^\bdot_{R'})$. Let $P^\blt$ be a $C^\blt_{L'}$-module.  We define
$$ X(P^\blt)_A = F_S(P^\blt_A) = P^\blt_A \otimes_{A^\blt} F_S^\blt \quad \text{and} \quad X(P^\blt)_B = P^\blt_B.$$
Starting with the structural morphism $P^\blt_B \tensor_{B^\blt} L^\blt[1] \to P^\blt_A$ we produce the structural morphism
$$ R' ( P^\blt_B ) \xrightarrow{\simeq} F_S( L'( P^\blt_B)) \to F_S ( P^\blt_A).$$
This is invertible because $F_S$ is an equivalence and we have an isomorphism $F_S^{-1} R' \to L'$.  

Consider the autoequivalence $\Psi_1 X \Psi_2$ of $\D(C^\bdot_{S'})$.  We observe by a straightforward computation that 
$$\Psi_1 X \Psi_2(\cB) = \cB'.$$
Now, we compute $\Psi_1 X \Psi_2(\cA)$.  First, $\Psi_2(\cA) \subset \D(C^\bdot_{L'})$ is the full subcategory of objects isomorphic to objects of the form $[ L'( S'(M_A^\bdot) ) \to M_A^\bdot]$, where the structure morphism is the counit of adjunction.  Next we compute that 
$$ X[ L'( S'(M_A^\bdot) ) \to M_A^\bdot ] = [R'(S'(M_A^\bdot)) \to F_S(M_A^\bdot)] $$
where the structure morphism is the composition of the map $R'( S'(M_A^\bdot) ) \to F_S L'( S'(M_A^\bdot))$ with the map $F_S ( L'S'(M_A^\bdot) ) \to F_S(M_A^\bdot)$ induced by the counit morphism.  This map is just the map coming from the triangle
$$ R' S' = RS \to F_S \to \id[1] \parr $$
defining $F_S$.  Therefore, after applying $\Psi_1$ we get 
$$ [ S'(M_A^\bdot)[1] \to S'(M_A^\bdot)[1] ] $$
where the structure morphism is the identity.  This is exactly the condition defining the category $\cA' = \Phi_2(\D(A^\bdot))$.  Thus $\D(C^\bdot_{S'})$ admits the fourth semi-orthogonal decomposition 
$$ \D(C^\bdot_{S'}) = \sod{\cA',\cB'}.$$
\end{proof}

\begin{rem}
If $S^\bdot, R^\bdot,$ and $L^\bdot$ are perfect bimodules then we may replace $\D(\bullet)$ with $\Perf(\bullet)$ in the above theorem.  If $A^\bdot$ and $B^\bdot$ are smooth and proper, then all cocontinuous functors between $\Perf(A^\bdot)$ and $\Perf(B^\bdot)$ are represesented by perfect bimodules.  
\end{rem}

\begin{rem}
There is an alternate formula for the twist.  Suppose that we have 
$$ \cC = \sod{\cA,\cB} = \sod{\cB,\cA'} = \sod{\cA',\cB'} = \sod{\cB',\cA} $$
as above.  Then $ T_S = i_{\cB}^L \circ \bL_{\cA}.$
(Compare with \ref{prop:wind_mut}, where $r$ plays the role of the quotient functor $i_{\cB}^L$.)
\end{rem}

%%%%%		%%%%%		%%%%%
%%%%%		%%%%%		%%%%%
%%%%%		%%%%%		%%%%%
%%%%%		%%%%%		%%%%%

\section{Monodromy of the quantum connection and fractional grade restriction rules}

In the remainder of this paper, we will refine the above construction of autoequivalences of $\D^b(X^{ss}/G)$ from a variation of GIT quotient. We generalize the grade restriction rules of Theorem \ref{thm:derived_Kirwan_surjectivity} in order to produce additional derived autoequivalences (see Corollary \ref{cor:factoring_twists_FEC}). Our motivation is to explain additional autoequivalences predicted by homological mirror symmetry (HMS). We first review Kontsevich's observation on how HMS leads to autoequivalences, as studied in \cite{ST01,Ho05,Ir}, then we frame these predictions in the context of variation of GIT quotient. We would like to emphasize that the following discussion of mirror symmetry is not meant to introduce new ideas of the authors -- we only hope to frame existing ideas regarding HMS in the context of GIT.

For simplicity we consider a smooth projective Calabi-Yau (CY) variety $V$ of complex dimension $n$. HMS predicts the existence of a mirror CY manifold $\hat{V}$ such that $D^b(V) \simeq D^b \op{Fuk} (\hat{V}, \beta)$, where $\beta$ represents a complexified K\"{a}hler class and $\D^b \op{Fuk}(\hat{V},\beta)$ is the graded Fukaya category. The category $\D^b \op{Fuk}$ does not depend on the complex structure of $\hat{V}$. Thus if $\hat{V}$ is one fiber in a family of compact CY manifolds $\hat{V}_t$ over a base $\cM$, the monodromy representation $\pi_1(\cM) \to \pi_0(\op{Symp}^{gr}(\hat{V},\beta))$ acting by symplectic parallel transport leads to an action $\pi_1(\cM) \to \op{Aut} \D^b \op{Fuk}(\hat{V},\beta)$. Via HMS this gives an action $\pi_1(\cM) \to \op{Aut} \D^b(V)$ (see \cite{ST01} for a full discussion).

Hodge theoretic mirror symmetry predicts the existence of a normal crossings compactification $\overline{\cM}(\hat{V})$ of the moduli space of complex structures on $\hat{V}$ along with a mirror map $\overline{\cM}(\hat{V}) \to \overline{\cK}(V)$ to a compactification of the ``complexified K\"{a}hler moduli space'' of $V$. Different regions of $\cK(V)$ correspond to different birational models of $V$, but locally $\cK(V)$ looks like the open subset of $H^2(V;\bC) / 2\pi i H^2(V;\bZ)$ whose real part is a K\"{a}hler class on $V$.\footnote{Technically, the complexified K\"{a}hler moduli space is locally $\cK(V) / \op{Aut}(V)$, but this distinction is not relevant to our discussion.} Mirror symmetry predicts that the mirror map identifies the \emph{$B$-model variation of Hodge structure} $H^n(\hat{V}_t)$ over $\cM$ with the $A$-model variation of Hodge structure, which is locally given by the quantum connection on the trivial bundle $\bigoplus H^{p,p}(V) \times \cK(V) \to \cK(V)$ (See Chapter 6 of \cite{CK} for details).

Finally, one can combine Hodge theoretic mirror symmetry and HMS: Let $\gamma: S^1 \to \cK(V)$ be the image of a loop $\gamma^\prime :S^1 \to \cM(\hat{V})$ under the mirror map. Symplectic parallel transport around $\gamma^\prime$ of a Lagrangian $L \subset \hat{V}_t$ corresponds to parallel transport of its fundamental class in the $B$-model variation of Hodge structure $H^n(\hat{V}_t)$. Thus mirror symmetry predicts that the automorphism $T_\gamma \in \op{Aut}(D^b(V))$ corresponds to the the monodromy of quantum connection around $\gamma$ under the twisted Chern character $\op{ch}^{2\pi i}$ defined in \cite{Ir}.

From the above discussion, one can formulate concrete predictions in the context of geometric invariant theory without an explicit mirror construction. For now we ignore the requirement that $V$ be compact (we will revisit compact CY's in Section \ref{subsect:complete_intersections}), and we restrict our focus to a small subvariety of the K\"{a}hler moduli space in the neighborhood of a ``partial large volume limit." Assume that $V = X^{ss}_- / G$ is a GIT quotient of a smooth quasiprojective $X$ and that $X^{ss}_-/G \dashrightarrow X^{ss}_+/G$ is a balanced GIT wall crossing with a single stratum and $\omega_X|_Z$ has weight $0$, as we studied in Section \ref{sect:mutations}.

The VGIT is determined by a $1$-parameter family of $G$-ample bundles $\cL_0 + r \cL^\prime$, where $r \in (-\epsilon, \epsilon)$. In fact we consider the two parameter space
$$U := \{ \tau_0 c_1(\cL_0) + \tau^\prime c_1(\cL^\prime) | \Re(\tau_0)>0 \text{ and } \Re(\tau^\prime) / \Re(\tau_0) \in (-\epsilon,\epsilon) \}$$
This is a subspace of $H^2(X^{ss}_- / G;\bC) / 2 \pi i H^2(X^{ss}_- / G; \bZ)$ obtained by gluing $\cK(X^{ss}_- / G)$ to $\cK(X^{ss}_+ / G)$ along the boundary where $\Re(\tau^\prime) = 0$. Because we are working modulo $2 \pi i \bZ$, it is convenient to introduce the exponential coordinates $q_0 = e^{-\tau_0}$ and $q^\prime = e^{-\tau^\prime}$. In these coordinates, we consider the partial compactification $\bar{U}$ as well as the annular slice $U_{q_0}$:
\begin{equation} \label{eqn:define_large_volume_limit}
\begin{array}{l}
\bar{U} := \left\{ (q_0,q^\prime) \in \bC \times \bC^\ast \left| |q_0|<1 \text{ and } |q^\prime| \in (|q_0|^\epsilon,|q_0|^{-\epsilon}) \right. \right\} \\

U_{q_0} := \{q_0\} \times \bC^\ast \cap \bar{U}.
\end{array}
\end{equation}

In this setting, mirror symmetry predicts that the quantum connection on $U$ converges to a meromorphic connection on some neighborhood of $U_0 = \{0\} \times \bC^\ast \subset \bar{U}$ which is singular along $U_0$ as well as a hypersurface $\nabla \subset \bar{U}$. To a path in $U \setminus \nabla$ connecting a point in the region $|q^\prime|<1$ with the region $|q^\prime|>1$, there should be an equivalence $\D^b(X^{ss}_-/G) \simeq \D^b(X^{ss}_+ / G)$ coming from parallel transport in the mirror family.

Restricting to $U_{q_0}$, one expects an autoequivalence of $\D^b(X^{ss}_-/G)$ for every element of $\pi_1(U_{q_0} \setminus \nabla)$, which is freely generated by loops around the points $\nabla \cap \bar{U}_{q_0}$ and the loop around the origin. We will refer to the intersection multiplicity of $\nabla$ with the line $\{0\} \times \bC^\ast$ as the \emph{expected number of autoequivalences produced by the wall crossing}. For a generic $q_0$ very close to $0$, this represents the number of points in $\nabla \cap U_{q_0}$ which remain bounded as $q_0 \to 0$.

For the example of toric CY manifolds, the compactification of the K\"{a}hler moduli space and the hypersurface $\nabla$ have been studied extensively. In Section \ref{sect:toric}, we compute these intersection multiplicities, which will ultimately inspire the construction of new autoequivalences of $\D^b(X^{ss}_- / G)$ in section \ref{sect:fract_shifts}.

\begin{rem} [Normalization]
In the discussion above, making the replacements $a \cL_0$ and $b \cL^\prime$ for positive integers $a,b$, and reducing $\epsilon$ if necessary, does not effect the geometry of the VGIT at all, but it replaces $U$ with the covering corresponding to the map $q_0 \mapsto q_0^a$, $q^\prime \mapsto (q^\prime)^b$. The covering in the $q_0$ direction has no effect on the expected number of autoequivalences defined above, but the covering $q^\prime \mapsto (q^\prime)^b$ would multiply the expected number of autoequivalences by $b$. Fortunately, the VGIT comes with a canonical normalization: When possible we will assume that $\cL^\prime|_Z \in \D^b(Z/L)_{1}$, and in general we will choose $\cL^\prime$ which minimizes the magnitude of the weight of $\cL^\prime|_Z$ with respect to $\lambda$. Multiplying $\cL_0$ if necessary, we can define the VGIT with $\epsilon = 1$.
\end{rem}

\begin{rem}
To simplify the exposition, we have ignored the fact that $X^{ss}_- / G$ is not compact in many examples of interest. To fix this, one specifies a function $W:X^{ss}_- / G \to \bC$ whose critical locus is a compact CY $V$, and the predictions above apply to the quantum connection of $V$ on the image of $\bar{U}$ under the map $H^2(\X^{ss}_-/G) \to H^2(V)$. We will discuss how autoequivalences of $\D^b(X^{ss}_- / G)$ lead to autoequivalences of $\D^b(V)$ in Section \ref{subsect:complete_intersections}.
\end{rem}

\begin{rem}
The region $U$ connects two large volume limits $q_0,q^\prime \to 0$ and $q_0,(q^\prime)^{-1} \to 0$. It is possible to reparameterize $U$ in terms of the more traditional large volume limit coordinates around either point (\cite{CK}, Chapter 6).
\end{rem}

%%%%%		%%%%%		%%%%%
%%%%%		%%%%%		%%%%%
%%%%%		%%%%%		%%%%%
%%%%%		%%%%%		%%%%%

\subsection{The toric case: K\"{a}hler moduli space and discriminant in rank 2}
\label{sect:toric}

A Calabi-Yau (CY) toric variety can be presented as a GIT quotient for a linear action of a torus $T \to \SL(V)$ on a vector space $V$ \cite{CLS}. Write $X^*(T)$ and $X_*(T)$ for the groups of characters and cocharacters of $T$, respectively. The GIT wall and chamber decomposition on $X^*(T)_\bR = X^*(T) \tensor \bR$ can be viewed as a fan known as the GKZ fan. The toric variety defined by this fan provides a natural compactification $\overline{\cK}$ of the complexified K\"{a}hler space $X^*(T) \otimes \bC^\ast$. A codimension-one wall in $X^*(T)_\bR$, which corresponds to a balanced GIT wall crossing, determines an equivariant curve $C \simeq \bP^1$ in $\overline{K}$ connecting the two large volume limit points determined by the chambers on either side of the wall. The curve $\bar{U}_0$ corresponding to this VGIT \eqref{eqn:define_large_volume_limit} is exactly the complement of the two torus fixed points in $C$.

Such a CY toric variety arises in mirror symmetry as the total space of a toric vector bundle for which a generic section defines a compact CY complete intersection (See Section \ref{subsect:complete_intersections}, and \cite{CK} for a full discussion). In this case, the toric variety defined by the GKZ fan also provides a natural compactification of the complex moduli space of the mirror $\overline{\cM}$. Although the mirror map $\overline{\cM} \to \overline{\cK}$ is nontrivial, it is the identity on the toric fixed points (corresponding to chambers in the GKZ fan) and maps a boundary curve connecting two fixed points to itself. It follows that our analysis of the expected number of autoequivalences coming from the VGIT should be computed in $\overline{\cM}$.

The boundary of $\overline{\cM}$, corresponding to singular complex degenerations of the mirror, has several components. In addition to the toric boundary, there is a particular hypersurface called the reduced discriminantal hypersurface $\nabla$ in $M$ (see \cite{GKZ}), which we simply call the \emph{discriminant}. It is the singular locus of the GKZ hypergeometric system. For simplicity we will analyze the case when $T$ is rank $2$. We will compute the expected number of autoequivalences as the intersection number between $C$ and the normalization of the discriminant. We will confirm turns out that this intersection number is equal to the length of a full exceptional collection on the $Z/L^\prime$ appearing in the GIT wall crossing, whenever $Z/L'$ is proper (see Section \ref{sect:background}).

Let $V = \bC^m$ and $(\bC^*)^2 \cong T \subset (\bC^*)^m$ be a rank two subtorus of the standard torus acting on $V$.  We can describe $T$ by a matrix of weights,
$$ \begin{pmatrix} a_1 & a_2 & \dotsm & a_m \\ b_1 & b_2 & \dotsm & b_m \end{pmatrix}, $$
representing the embedding $(t,s) \mapsto (t^{a_1}s^{b_1},\dotsc,t^{a_m}s^{b_m})$. We assume that all columns are non-zero. The CY condition means that we have
$$ \sum_{i=1}^m{a_i} = \sum_{i=1}^m{b_i} = 0. $$
Now, up to an automorphism of $V$ we may assume that the matrix of weights has the following form
$$ \begin{pmatrix} a_i \\ b_i \end{pmatrix} = \begin{pmatrix} d^1_1 \chi_1 & \dotsm & d^1_{n_1} \chi_1 & \dotsm & d^r_1 \chi_r & \dotsm & d^r_{n_r} \chi_r \end{pmatrix} $$
where $\chi_j = \bigl( \begin{smallmatrix} \alpha_j \\ \beta_j \end{smallmatrix} \bigr)$ and $\chi_1,\dotsc \chi_r$ are ordered counterclockwise by the rays they generate in the plane. Using the fact that a wall between GIT chambers occurs when there exists a strictly semistable point, one can determine that the rays of the GKZ fan are spanned by $-\chi_j$.  The GIT chambers, the maximal cones of the GKZ fan, are the cones $\sigma_i = cone(-\chi_i,-\chi_{i+1})$, $i < r$, and $\sigma_r = ( -\chi_r, -\chi_1 )$. The discriminant admits a rational parameterization, called the Horn uniformization, $f:\bP(X_*(T)_{\bC}) \dashrightarrow \nabla$ of the following form.  Set $d^i = \sum_{j=1}^{n_i}{d^i_j}$.  For a Laurent monomial $x^\lambda \in \bC[T]$ we have
$$ f^*(x^{\lambda}) = \prod_{i,j}{ (d^i_j \chi_i)^{-d^i_j (\chi_i,\lambda)}} = d_{\lambda} \prod_i \chi_i^{-d^i(\chi_i,\lambda)}, \quad d_{\lambda} := \prod_{i,j}{ (d^i_j)^{-d^i_j(\chi_i,\lambda)} } $$ 
where we view $X^*(T)$ as a set of linear functions on $X_*(T)_{\bC}$.  It follows from the CY condition that $f^*(x^{\lambda})$ has degree zero as a rational function on $X_*(T)_{\bC}$ and that $\cM$ is proper.  Therefore, $f$ actually defines a regular map $\bP(X_*(T)_\bC) \cong \bP^1 \to \cM$. We define $C_i$ to be equivariant curve in $\overline{\cM}$ defined by the codimension one wall $\bR_{\geq 0} \cdot (-\chi_i)$.

\begin{prop}
If $-\chi_i$ is not among the $\chi_j$, then the length of $\bP(X_*(T)_\bC) \times_\cM C_i$ is $d^i$.
\end{prop}

\begin{proof}
$C_i$ is covered by the open sets corresponding to $\sigma_{i-1}$ and $\sigma_i$.  Let $U_i$ be the chart corresponding to $\sigma_i$.  Recall that the coordinate ring of $U_i$ is
$$\bC[ \sigma_i^\vee] = \bC\{ x^\lambda : \forall \chi \in \sigma_i, (\chi, \lambda) \geq 0 \} \subset \bC[X_*(T)]. $$
Observe that $(\chi,\lambda) \geq 0$ for all $\chi \in \sigma_i$ if and only if $(\chi_i,\lambda), (\chi_{i+1},\lambda) \leq 0$.  Next, we must compute the ideal of $C_i$ in the charts $U_i$ and $U_{i+1}$.  Since $-\chi_i$ spans the wall under consideration the ideal of $C_i$ will be
$$ I_i = \bC\{x^\lambda : (\chi_i, \lambda) < 0 \} \cap \bC[U_\bullet]. $$
Let $p_j = \{ \chi_j = 0\} \in \bP(X_*(T)_\bC)$.  Then $f(p_j) \in U_i$ if and only if for all $\lambda$ such that $(\chi_i,\lambda),(\chi_{i+1} ,\lambda) \leq 0$ we have $(\chi_j,\lambda) \leq 0$.   So if $\chi_j \neq \chi_i, \chi_{i+1}$ then $f(p_j) \notin U_i$ and $f^{-1}( U_i \cap \nabla )$ is supported on $\{p_i,p_{i+1}\}$.

Then there clearly exists a $\lambda$ such that $(\chi_{i+1},\lambda) = 0$ but $(\chi_i, \lambda) < 0$.  This means that in fact $f^{-1}(U_i \cap \nabla)$ is supported on $p_i$.  So we can compute the length of $f^{-1}(U_i \cap \nabla)$ after restricting to $\bP(X_*(T)_\bC) \setminus \{p_j\}_{j \neq i}$ where its ideal is generated by $\{ \chi_i^{-d_i(\chi_i,\lambda)} \}_{\lambda \in \sigma_i^\vee}$.  Finally, we note that
$$ \min\{ (\chi_i,\lambda) : \lambda \in \sigma_i^\vee \} = 1 $$
and therefore the length of $f^{-1}(\nabla \cap U_i )$ is $d^i$.  By an analogous argument we see that $f^{-1}(\nabla \cap U_i ) = f^{-1}(\nabla \cap U_{i-1} )$.
\end{proof}

\begin{rem}
Observe that the image of $f$ avoids the torus fixed points.  Indeed, the torus fixed point in $U_i$ lies on $C_i \setminus U_{i-1}$, but $\nabla \cap C_i \subset U_i \cap U_{i-1}$.
\end{rem}

Codimension one wall crossings are always balanced \cite{DH}, but we include the analysis of the Hilbert-Mumford numerical criterion in order to explicitly identify the $Z/L^\prime$ when we cross the wall spanned by $-\chi_i$ where $-\chi_i$ is not also a weight of $T$ acting on $V$. For any character the KN stratification is determined by data $\{ (Z_j, \lambda_j) \}_{j=0}^r$ (see Section \ref{sect:background}).

\begin{prop}
Let $\{(Z_j^R, \lambda^R_j)\}_{j=0}^l$ and $\{ (Z_j^L,\lambda^L_j) \}_{j=0}^s$ be the data of stratifications immediately to the right and left of the wall spanned by $-\chi_i$, respectively.  Then 
\begin{enumerate}
	\item $\lambda^R_0 = -\lambda^L_0$ and $(\chi_i,\lambda^\bullet_0) = 0$,
	\item $Z^R_0 = Z^L_0 = V^{\lambda_0} \setminus 0$, and
	\item $\bigcup_{j > 0}{ S^R_j } = \bigcup_{j > 0}{S^L_j} $.
\end{enumerate}
\end{prop}
\begin{proof}[Proof sketch. (See \cite{DH} for details.)]
Let $\chi$ be a character near $-\chi_i$ (as rays),  $\|\cdot\|$ be a norm on $X_*(T)_\bR$, and $\mu^{\chi}(\lambda) = \frac{(\chi,\lambda)}{\| \lambda\| }$.  In this situation the KN stratification is defined inductively.  First, there is a primitive cocharacter $\lambda_{max}$ which maximizes $\mu^{\chi}$.  The most unstable stratum has core $Z_{max} = V^{\lambda_{max}}= 0$ and $S_{max} = \oplus_{ i, (\chi_i,\lambda_{max}) \geq 0 }  \oplus_j V_{i,j}$.  The linearization $\chi$ determines a choice of generator for the line perpendicular to $\chi_j$.  For each $j$ we let $\lambda_j$ be the primitive cocharacter satisfying (i) $(\chi_j,\lambda_j) = 0$, and (ii) $\mu^\chi(\lambda) \geq 0$.  We arrange these in decreasing order according to the value of $\mu^\chi(\lambda_\bullet)$: $\lambda_{j_1}, \dotsc, \lambda_{j_n}$.  If $V^{\lambda_{j_k}}$ is not entirely contained in $S_{< k} = S_{max} \cup \bigcup_{i < k} S_i$ then we put $Z_k = V^{\lambda_{j_k}} \setminus S_{< k}$ and $S_k = \big(\oplus_{i, (\chi_i,\lambda_{j_k})\geq 0} \oplus_j V_{i,j} \big) \setminus S_{< k}$.  Clearly then, the KN stratification only depends on the sequence of $\lambda_\blt$.  Now, as $\chi$ varies across the wall, $\lambda_{max}$ varies, but $Z_{max}$ and $S_{max}$ remain unchanged.  Furthermore, $\mu^\chi(\lambda_j)$ remains positive unless $j=i$ and moreover the ordering on $\lambda_j$ for $j \neq i$ does not change.  On the other hand $\mu^\chi(\lambda_i)$ changes sign so that $-\lambda_i$ replaces $\lambda_i$ as the cocharacter attached to the least unstable stratum.  The proposition follows.
\end{proof}

Note that $V^{\lambda^\pm_0} = \oplus_j{ V_{i,j} }$.  The action of $T$ on $V^{\lambda_0}$ factors through $\chi_i$ and the weights are simply $d^i_1,\dotsc, d^i_{n_i}$ which are all positive.  Therefore the stack $Z_0 / \bC^*$ is a weighted projective space.  Its derived category is understood thanks to the following.

\begin{thm*}[Theorem 2.12 of \cite{AKO}]
$\D^b(\bP(d^i_1,\dotsc,d^i_{n_i}))$ has a full exceptional collection of $d^i$ line bundles.  In particular $K_0(\bP(d^i_1,\dotsc,d^i_{n_i}))$ is free of rank $d^i$.
\end{thm*}

In conclusion, we see that the length of a full exceptional collection on $Z_0/L_0$ associated to a wall $i$ is equal to the intersection multiplicity of $f:\bP^1 \to \nabla$ with the curve $C_i$.

\begin{rem}
In higher rank, our naive method is insufficient to handle the combinatorics of the discriminant.  However in light of the discussion following Corollary \ref{cor:factoring_twists_FEC}, we can make a precise prediction about how the discriminant intersects certain curves at infinity.  Given an action of a torus $T \to \SL(V)$ we diagonalize it so that $T \ni t \mapsto (t^{\chi_1},\dotsc,t^{\chi_n})$ where $n = \dim(V)$ and $\chi_1,\dotsc,\chi_n$ are characters of $T$.  For each $\chi \in X^*(T)$ let $|\chi| = \max\{d \in \bN : \chi/d \in X^*(T) \}$ and for each $\lambda \in X_*(T)$ define 
\[ d(\lambda) = \sum_{i : \chi_i(\lambda)=0}{|\chi_i|}. \]

Variation of GIT for $T$ acting on $V$ defines a fan $\Sigma$ in $X^*(T)_\bR$ and the Kirwan-Ness-Hesselink stratification assigns to each codimension one walls a pair $\lambda^{\pm 1} \in X_*(T)$ of one parameter subgroups.  Fix a codimension one wall $W$ and its pair $\lambda^{\pm 1}$ of one-parameter subgroups.  If the set $\{\chi_i: \chi_i(\lambda) = 0 \}$ is contained in an open half space of $X^*(T)_\bR$ then $Z^{\lambda}/L^\prime(\lambda)$ will be a weighted projective space with a full exceptional collection of $d(\lambda)$ line bundles (we drop the superscript $\pm 1$ from $\lambda^{\pm 1}$ since it does not matter here).  Now if $C$ is the  curve in $X_\Sigma$ corresponding $W \in \Sigma$ then we expect that $\ell(\nabla \cap C) = d(\lambda)$.
\end{rem}

\begin{ex}
Consider the $T = (\bC^*)^2$ action on $\bA^8$ given by
$$ (t,s) \mapsto (t,t,t,s,s,s,t^{-2},t^{-1}s^{-3}). $$
The wall and chamber decomposition of $\bR^2$ associated to this action is given in the following diagram.
\begin{figure}[h]
\centering
\begin{tikzpicture}[x=.5mm, y=.5mm, inner xsep=0pt, inner ysep=0pt, outer xsep=0pt, outer ysep=0pt]

%% create grid & axes
\draw[style={gray, very thin},step=10] (-40,-40) grid (40,40);

%% walls
\draw[line width=.5mm] (0,0)--(-40,0);
\draw[line width=.5mm] (0,0)--(0,-40);
\draw[line width=.5mm] (0,0)--(40,0);
\draw[line width=.5mm] (0,0)--(13.33,40);

\node at (-46,0) {$W_2$};
\node at (0,-45) {$W_1$};
\node at (13.33,45) {$W_3$};
\node at (46,0) {$W_4$};

%% chambersp
\node at (-35,-35) {I};
\node at (-35,35) {II};
\node at (35,35) {III};
\node at (35,-35) {IV};

\end{tikzpicture}
\end{figure}

Chamber I corresponds to the total space of $\cO(-2,0) \oplus \cO(-1,-3)$ over $\bP^2 \times \bP^2$, and for this reason we will return to this example in subsection \ref{subsect:complete_intersections}.  By Horn uniformization, the discriminant is parameterized by
$$ [u:v] \mapsto \big(-4\frac{u+3v}{u}, -\frac{(u+3v)^3}{v^3}\big). $$
We will compute the intersection number at wall $W_3$.  This corresponds to the character $(-1,-3)$.  No other characters are a rational multiple of this one.  Therefore, we should get intersection number 1.  We compute the dual cones to chambers II and III, and indicate the ideal of $C_3$ in the diagram below.  The nested grey regions correspond to the monomials in the dual cones and in the ideal of $C_3$.  The vertical, horizontal, and diagonal lines divide the plane into regions corresponding to monomials where $u$,$v$, and $(u+3v)$ respectively appear with positive or negative exponents.  It is clear that only $(u+3v)$ always appears with a positive exponent.  It appears in the first coordinate with exponent 1 and therefore the intersection number $\ell(C_3 \cap \nabla)$ is one.
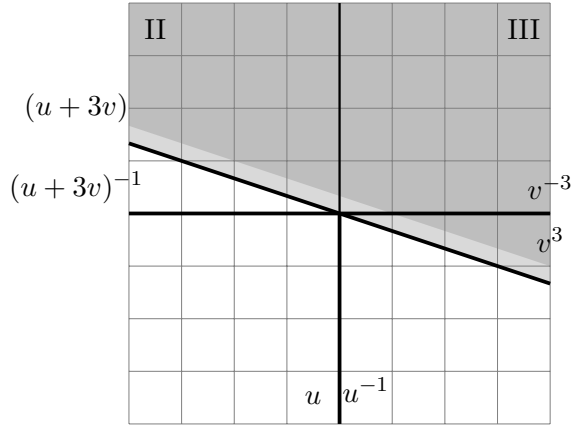
\begin{figure}[h]
\centering
\begin{tikzpicture}[x=.7mm, y=.7mm, inner xsep=0pt, inner ysep=0pt, outer xsep=0pt, outer ysep=0pt]

%% create grid & axes
\draw[style={gray, very thin},step=10] (-40,-40) grid (40,40);

\fill[fill=gray, fill opacity = 0.3] (-40,13.33) -- (-40,40)-- (40,40) -- (40,-13.33) -- cycle;
\fill[fill = gray, fill opacity = 0.3] (-40,50/3) -- (40,-10) -- (40,40) -- (-40,40) -- cycle;

%% chambers
\node at (-35,35) {II};
\node at (35,35) {III};

\draw[line width=0.5mm] (0,-40)--(0,0);
\draw[line width=0.5mm] (-40,0)--(40,0);
\draw[line width=0.5mm] (-40,13.33) -- (40,-13.33);
\draw[line width=0.3mm] (0,0)--(0,40);

\node at (-5,-35) {$u$};
\node at (5,-33.5) {$u^{-1}$};
\node at (40,5) {$v^{-3}$};
\node at (40,-5) {$v^3$};
\node at (-50,20) {$(u+3v)$};
\node at (-50,5) {$(u+3v)^{-1}$};

\end{tikzpicture}
\caption{The dual cones to chambers II and III and the ideal defining $C_3$.}
\end{figure}
Similar analysis of walls $W_1,W_2$ and $W_4$ gives intersection numbers 3, 1, and 1, respectively.  The intersection number $\ell(C_1 \cap \nabla)$ agrees with the quantity computed from the characters. However, the intersection multiplicities $\ell(C_2 \cap \nabla)$ and $\ell(C_4 \cap \nabla)$ do not.  As we compute in Example \ref{ex:K3}, the locus $Z/L^\prime$ associated to the wall $W_2$ is $\op{tot} \cO_{\bP^2}(-2) / \bC^*$.  Note that its derived category does not admit a full exceptional collection.
\end{ex}

%%%%%		%%%%%		%%%%%
%%%%%		%%%%%		%%%%%
%%%%%		%%%%%		%%%%%
%%%%%		%%%%%		%%%%%

\subsection{Fractional grade restriction rules} \label{sect:fract_shifts}

In order to construct additional derived equivalences, we introduce \emph{fractional grade restriction rules} given a semiorthogonal decomposition $\D^b(Z/L)_w = \langle \cE_0, \cE_1\rangle$, the data of which we will be denoted $\mathit{e}$. This will be of particular interest when $\D^b(Z/L)_w$ has a full exceptional collection.

The equivalence of Lemma \ref{lem:unstable_category} gives a semiorthogonal decomposition $\cA^+_w = \langle \cE_0^+,\cE_1^+ \rangle$, where $\cE^+_i = \iota_w(\cE_i)$. We can refine the semiorthogonal decompositions \eqref{eqn:two_semi_decomps}
$$\cC^+_w = \langle \cE^+_0,\cE^+_1, \G^+_{w+1} \rangle = \langle \G^+_w, \cE^+_0,\cE^+_1 \rangle$$
Because $\cE_0^+$ and $\cE_1^+$ are left and right admissible in $\cC^+_w$ respectively, we can make the following
\begin{defn} Given the semiorthogonal decomposition $e$, we define the full subcategory $\G_e^+ = (\cE^+_1)^\perp \cap ^\perp\!(\cE^+_0) \subset \cC^+_w$. In other words, it is defined by the semiorthogonal decomposition
$$\cC^+_w = \langle \cE^+_0, \G^+_e, \cE^+_1 \rangle$$
Because $\cE_0^+$, and $\cE_1^+$ generate the kernel of the restriction functor $r_+$, it follows formally that $r_+ : \G^+_e \to \D^b(X^{ss}_+/G)$ is an equivalence of dg-categories.
\end{defn}

The mutation equivalence functor factors
$$\xymatrix{\G^+_{w+1} \ar[r]_{\lmut_{\cE^+_1}} \ar@/^/[rr]^{\lmut_{\cA^+_w}} & \G^+_e \ar[r]_{\lmut_{\cE^+_0}} & \G^+_w }$$
In order for these intermediate mutations to induce autoequivalences of $\D^b(X^{ss}_- / G)$, we must show that $\G^+_e$ is also mapped isomorphically onto $\D^b(X^{ss}_- / G)$ by restriction. We let $\kappa_\pm$ denote the equivariant line bundle $\det (N_{S^\pm} X)|_Z) = (j^\pm)^!\cO_X|_Z$.

\begin{lem} \label{lem:characterize_refined_window}
Let $F^\bdot \in \cC^+_w$.  Then $F^\bdot \in \G^+_e$ if and only if it satisfies the ``\emph{fractional grade restriction rule}":
\begin{equation}\label{eqn:fractional_grade_restriction_rule}
(\sigma^\ast F^\bdot)_w \in {^\perp}(\cE_0) \quad \text{and} \quad (\sigma^\ast F^\bdot \otimes \kappa_+)_w \in (\cE_1)^\perp
\end{equation}
\end{lem}

\begin{proof}
By definition $F^\bdot \in \G^+_e$ if and only if $\op{Hom}(F^\bdot, \iota_w (\cE_0)) = 0$ and $\op{Hom}(\iota_w(\cE_1),F^\bdot) = 0$. By Lemma \ref{lem:adjoint_unstable}, the left and right adjoint of $\iota_w$ can be expressed in terms of $\sigma^\ast F^\bdot$. We use that $(\sigma^\ast F^\bdot)_{w+\eta} \otimes \kappa_+ = (\sigma^\ast F^\bdot \otimes \kappa_+)_w$.
\end{proof}

One can think of $\G^+_e$ as a refined version of the usual category $\G^+_w$. Previously, we had an infinite semiorthogonal decomposition $\D^b(Z/L) = \sod{\ldots,\D^b(Z/L)_w,\D^b(Z/L)_{w+1},\ldots}$, and the grade restriction rule amounted to choosing a point at which to split this semiorthogonal decomposition, then requiring $\sigma^\ast F^\bdot$ to lie in the right factor and $\sigma^\ast F^\bdot \otimes \kappa_+$ to lie in the left factor. Lemma \ref{lem:characterize_refined_window} says the same thing but now we use the splitting
$$\D^b(Z/L) = \sod{\sod{\ldots, \D^b(Z/L)_{w-1}, \cE_0},\sod{\cE_1,\D^b(Z/L)_{w+1},\ldots}}.$$

The canonical bundle for a quotient stack $Z/L$ is $\omega_{Z/L} = \omega_Z \otimes \det \lie{l}^\dual$. \footnote{This is the same as $\omega_Z$ if $L$ is connected.} We say that \emph{Serre duality holds for $Z/L$} if the category $\D^b(Z/L)$ is Hom-finite and $\otimes \omega_{Z/L}[n]$ is a Serre functor for some $n$, i.e. $\op{Hom}^\bdot_{Z/L} (F^\bdot , G^\bdot \otimes \omega_{Z/L}[n]) \simeq \op{Hom}^\bdot_{Z/L} (G^\bdot, F^\bdot)^\dual$. Because all objects and homomorphism split into direct sums of weights spaces for $\lambda$, and $\omega_{Z/L} \in \D^b(Z/L)_0$, this is equivalent to Serre duality holding in the subcategory $\D^b(Z/L)_0 \simeq \D^b(Z/L^\prime)$. Thus whenever $Z/L^\prime$ is a compact DM stack, Serre duality holds for $Z/L$.

\begin{prop} \label{prop:refined_wall_cross}
Let $\omega_{X/G}|_Z \simeq \cO_Z$, and assume that Serre duality holds for $Z/L$, then $r_- : \G^+_e \to \D^b(X^{ss}_- / G)$ is an equivalence of dg-categories. More precisely $\G^+_e = \G^-_{e^\prime}$, where $e^\prime$ denotes the data of the semiorthogonal decomposition
$$\D^b(Z/L)_{[\lambda^- = w^\prime]} = \langle \cE_1 \otimes \omega_{Z/L} \otimes \kappa^\dual_+, \cE_0 \otimes \kappa^\dual_+ \rangle$$
\end{prop}

\begin{proof}
First note that $\mathit{e}^\prime$ is actually a semiorthogonal decomposition by Serre duality: it is the left mutation of $e$ tensored with $\kappa_+^\dual$.

Applying Serre duality to the characterizaton of $\G^+_e$ in Lemma \ref{lem:characterize_refined_window}, and using the fact that $(\bullet)_{[\lambda^+=w]} = (\bullet)_{[\lambda^-=w^\prime + \eta]}$, it follows that $F^\bdot \in \G^+_e$ if and only if
$$(\sigma^\ast F^\bdot \otimes \kappa_- )_{[\lambda^-=w^\prime]} \in (\cE_0 \otimes \omega_{Z/L}^\dual \otimes \kappa_-)^\perp \quad \text{and} \quad (\sigma^\ast F^\bdot)_{[\lambda^- = w^\prime]} \in {^\perp}(\cE_1 \otimes \omega_{Z/L} \otimes \kappa_+^\dual)$$
This is exactly the characterization of $\G^-_{e^\prime}$, provided that $\kappa_- \otimes \omega_{Z/L}^\dual \simeq \kappa_+^\dual$.

Consider the weight decomposition with respect to $\lambda^+$, $\Omega^1_X|_Z = (\Omega^1_X)_+ \oplus (\Omega^1_X)_{0} \oplus (\Omega^1_X)_{-}$. Then $\omega_{Z/L} \simeq \det ((\Omega^1_X)_0) \otimes \det(\lie{g}_0)^\dual$, and $\kappa_\pm^\dual \simeq \det ((\Omega^1_X)_\pm) \otimes \det (\lie{g}_{\pm})^\dual$, where $\lie{g}_\pm$ denotes the subspace of $\lie{g}$ with positive or negative weights under the adjoint action of $\lambda^+$. Hence $\omega_{X/G}|_Z \simeq \kappa_+^\dual \otimes \omega_{Z/L} \otimes \kappa_-^\dual$, so when $\omega_{X/G}|_Z \simeq \cO_Z$ we have $\kappa_+^\dual \simeq \kappa_- \otimes \omega_Z^\dual$ as needed.
\end{proof}

\begin{cor} \label{cor:multiple_refined_windows}
Let $\omega_{X/G}|_Z \simeq \cO_Z$ equivariantly, let Serre duality hold for $Z/L$, and assume we have a semiorthogonal decomposition $\D^b(Z/L)_w = \langle \cE_0, \ldots, \cE_N \rangle$. If we define $\cH^+_i$ as the mutation $\cC_w =\langle \cE_0^+, \ldots, \cH^+_i,\cE^+_i,\ldots,\cE^+_N \rangle$, then $r_- : \cH^+_i \to \D^b(X^{ss}_-/G)$ is an equivalence.
\end{cor}
\begin{proof}
Apply Proposition \ref{prop:refined_wall_cross} to the two term semiorthogonal decomposition $\langle \cA_0,\cA_1 \rangle$, where $\cA_0 = \langle \cE_0,\ldots,\cE_{i-1}\rangle$ and $\cA_1 = \langle \cE_i, \ldots, \cE_N \rangle$.
\end{proof}

As a consequence of Corollary \ref{cor:multiple_refined_windows} and the results of Section \ref{sect:mutations}, one can factor the window shift $\Phi_w$ as a composition of spherical twists, one for each semiorthogonal factor $\cE_i$. For concreteness, we narrow our focus to the situation where $\D^b(Z/L)_w$ admits a full exceptional collection $\langle E_0,\ldots,E_N\rangle$. In this case the $\cE_i^+$ of Corollary \ref{cor:multiple_refined_windows} are generated by the exceptional objects $E_i^+ := j^+_\ast (\pi^+)^\ast E_i$. The category $\cH^+_i$ is characterized by the fractional grade restriction rule
\begin{equation} \label{eqn:refined_GRR}
\begin{array}{l} \op{Hom}_{Z/L} ((\sigma^\ast F^\bdot)_{[\lambda^+ = w]}, E_j) = 0, \text{ for } j<i, \text{ and } \\ \op{Hom}_{Z/L} \left(E_j, (\sigma^\ast F^\bdot \otimes \kappa_+)_{[\lambda^+ = w]} \right) = 0, \text{ for } j\geq i \end{array}
\end{equation}

\begin{cor} \label{cor:factoring_twists_FEC}
Let $\omega_{X/G} |_Z \simeq \cO_Z$ and let $\D^b(Z/L)_w = \sod{E_0,\ldots,E_N}$ have a full exceptional collection. Then the objects $S_i := f_w(E_i) = j^+_\ast (\pi^+)^\ast E_i|_{X^{ss}_-} \in \D^b(X^{ss}_- / G)$ are spherical, and $\Phi_w = T_{S_0} \circ \cdots \circ T_{S_N}$.
\end{cor}

As noted, this follows for purely formal reasons from Corollary \ref{cor:multiple_refined_windows} and the results of subsection \ref{subsect:factorization}, but for the purposes of illustration we take a more direct approach.

\begin{proof}
We use Lemma \ref{lem:local_cohomology} and the fact that $(\sigma^\ast E_i^+)_{[\lambda^-=w^\prime]} = (\sigma^\ast E^+_i)_{[\lambda^+ = w+\eta]} = E_i \otimes \kappa_+^\dual$ to compute
\begin{align*}
R\Gamma_{S^-} \inner{\op{Hom}}(E^+_i, F^\bdot) &\simeq \op{Hom}_{Z/L}(E_i, \sigma^\ast(F^\bdot)_{w^\prime+\eta} \otimes \kappa_- \otimes \kappa_+) \\
&\simeq \op{Hom}_{Z/L}(E_i, \sigma^\ast(F^\bdot)_{w^\prime+\eta} \otimes \omega_{Z/L} \otimes \omega_{X/G}^{-1})
\end{align*}
Now let $\omega_{X/G} \simeq \cO_Z$. Serre duality implies that $$\op{Hom}_{Z/L}(E_i, \sigma^\ast(F^\bdot)_{w^\prime+\eta} \otimes \omega_Z) = \op{Hom}_{Z/L} (\sigma^\ast(F^\bdot)_{w^\prime+\eta}, E_i)^\dual.$$
Thus by \eqref{eqn:refined_GRR}, the canonical map $\op{Hom}_{X/G} (E_i^+, F^\bdot) \to \op{Hom}_{X^{ss}_-/G}(S_i, F^\bdot|_{X^{ss}_-})$ is an isomorphism for $F^\bdot \in \cH^+_{i+1}$. This implies the commutative diagram
$$\xymatrix{\G^+_{w+1} \ar[r]^{L_{E^+_N}} \ar[d]^{r^-} & \cH^+_N \ar[r]^{L_{E^+_{N-1}}} \ar[d]^{r^-} & \cdots \ar[r]^{L_{E^+_1}} & \cH^+_1 \ar[r]^{L_{E^+_0}} \ar[d]^{r^-} & \G^+_w \ar[d]^{r^-} \\ \D^b(X^{ss}_- / G) \ar[r]^{T_{S_N}} & \D^b(X^{ss}_- / G) \ar[r]^{T_{S_{N-1}}} & \cdots \ar[r]^{T_{S_1}} & \D^b(X^{ss}_- / G) \ar[r]^{T_{S_0}} & \D^b(X^{ss}_- / G) }$$
Where $T_{S_i}$ is the twist functor $\op{Cone}(\op{Hom}(S_i,F^\bdot) \otimes S_i \to F^\bdot)$. By \ref{prop:refined_wall_cross}, the functors $r^-$ are equivalences, and therefore so are $T_{S_i}$.
\end{proof}

Corollary \ref{cor:factoring_twists_FEC}, suggests a natural interpretation in terms of monodromy as discussed in the beginning of this section. Let $U_{q_0}$ be the annulus \eqref{eqn:define_large_volume_limit}, with $|q_0|$ small, and let $p_0,\ldots,p_N$ be the points of $U_{q_0} \setminus \nabla$ which remain bounded as $q_0 \to 0$. Consider an ordered set of elements $[\gamma_0],\ldots,[\gamma_N]$ of $\pi_1(U_{q_0} \setminus \nabla)$ such that 
\begin{enumerate} \label{eqn:winding_condition}
\item $\gamma_i$ lie in a simply connected domain in $U_{q_0}$ containing $p_0,\ldots,p_N$, and
\item there is a permutation $\sigma_i$ such that the winding number of $\gamma_i$ around $p_j$ is $\delta_{j,\sigma_i}$.
\end{enumerate}
It is natural to guess that the monodromy representation $\pi_1(U_{q_0} \setminus \nabla) \to \op{Aut} \D^b(X^{ss}_- / G)$ predicted by mirror symmetry assigns $T_{S_i}$ to $[\gamma_i]$. In particular, it would be interesting to compare the monodromy of the quantum connection with the action of $T_{S_i}$ under the twisted Chern character.

\begin{figure}[h] \label{fig:monodromy}
\begin{center}
%\begin{tikzpicture}[x=1mm, y=1mm, inner xsep=0pt, inner ysep=0pt, outer xsep=0pt, outer ysep=0pt]
%
%\draw (0,0) circle (.7mm);
%
%\draw[dashed] (0,0) circle (5mm);
%\node[anchor={east}] at (-3,3) {\small $|q^\prime| = 1$};
%this
%
%\draw[line width=.3mm, dashed, <->] (0,0) -- (30,0);
%
%\foreach \x / \y / \i in {12/-5/0, 14/2/1, 10/7/2}
%{
%
%	\fill [fill=gray] (\x,\y) circle (.7mm);
%	\draw (4,0) to[out=0,in=180] (\x,\y-2) to[out=0,in=270, ->] ++(2,2) to[out=90,in=0] ++(-2,2) to[out=180,in=0] (4,0);
%	\node[anchor={west}] at (\x+3,\y) {\small $T_{S_\i}$};
%
%}
%
%\draw (4,0) .. controls +(-45:5mm) and +(180:5mm) .. (12,-9)
%	.. controls +(0:5mm) and +(270:5mm) .. (24,0)
%	.. controls +(90:5mm) and +(0:5mm) .. (10,11)
%	.. controls +(180:5mm) and +(45:5mm) .. (4,0);
%\node[anchor={west}] (L1) at (24,5) {$\Phi_w$};
%\fill (4,0) circle (.5mm);
%	
%\end{tikzpicture}
\includegraphics[scale=.4]{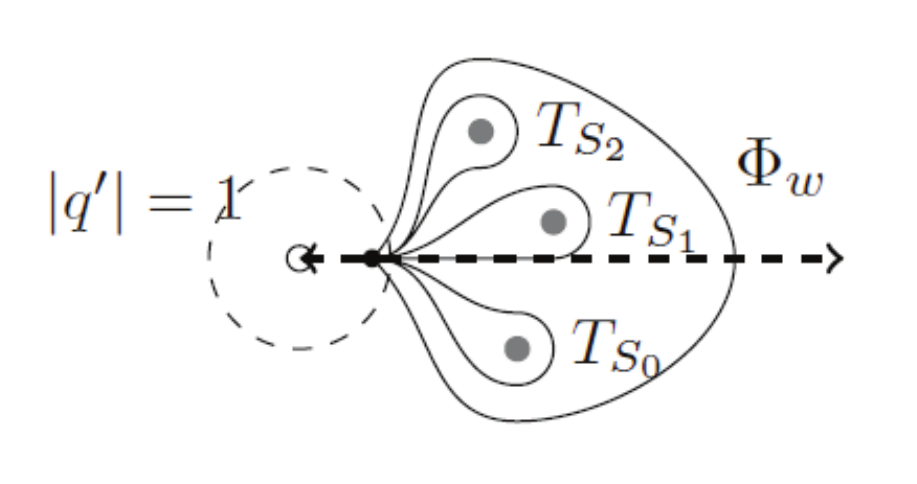}
\end{center}
\caption{Loops in $U_{q_0} \setminus \nabla$ corresponding to monodromy of the quantum connection of $X^{ss}_- / G$, giving a pictorial interpretation of Corollary \ref{cor:factoring_twists_FEC}.}
\end{figure}

Evidence for this interpretation of Corollary \ref{cor:factoring_twists_FEC} is admittedly circumstantial. In \ref{sect:toric}, we verified that the number of autoequivalences predicted by mirror symmetry is the same as the length of a full exceptional collection on $\D^b(Z/L^\prime)$ for toric flops of CY toric varieties of Picard rank 2. Letting $q_0 \to 0$, the points $p_0,\ldots,p_n$ converge to $1 \in U_0$. Horja \cite{Ho05} studied the monodromy of the quantum connection and the corresponding autoequivalences for the boundary curve $U_0$, and his work can be used to verify our interpretation of the loop corresponding to $\Phi_w$.

Furthermore, if we fix a simply connected domain $D \subset U_{q_0}$ containing $p_0,\ldots,p_N$ and let $\op{Diff}(D \setminus \{p_0,\ldots,p_N\},\partial D)$ denote the topological group of diffeomorphisms which restrict to the identity on the boundary, then $B_{N+1} \simeq \pi_0 \op{Diff}(D\setminus \{p_0,\ldots,p_N\}, \partial D)$ is a braid group which acts naturally on ordered subsets of $\pi_1(U_{q_0})$ satisfying \eqref{eqn:winding_condition}. The braid group also acts formally by left and right mutations on the set of full exceptional collections $\D^b(Z/L)_w = \sod{E_0,\ldots,E_N}$, and these two actions are compatible (See Figure \ref{fig:braiding}).

\begin{figure}[h] 
\caption{Dictionary between action of $B_3$ on loops in $D\setminus \{p_0,p_1,p_2\}$ and on full exceptional collections of $\D^b(Z/L)_w$.\label{fig:braiding}}
\begin{subfigure}[h]{.4\textwidth}
\centering
%\begin{tikzpicture}[x=1mm, y=1mm, inner xsep=0pt, inner ysep=0pt, outer xsep=0pt, outer ysep=0pt]
%
%\draw (0,0) circle (.7mm);
%\draw[line width=.3mm, dashed, <->] (0,0) -- (25,0);
%\draw[gray] (13,0) circle [x radius = 9, y radius = 12];
%\node[anchor={east}, gray] at (5,10) {$\partial D$};
%
%\foreach \x / \y / \i in {12/-5/2, 14/2/1, 10/7/0}
%{
%
%	\fill [fill=gray] (\x,\y) circle (.7mm);
%	\draw (4,0) to[out=0,in=180] (\x,\y-2) to[out=0,in=270, ->] ++(2,2) to[out=90,in=0] ++(-2,2) to[out=180,in=0] (4,0);
%	\node[anchor={west}] at (\x+3,\y) {$\gamma_\i$};
%
%}	
%\end{tikzpicture}
\includegraphics[scale=.4]{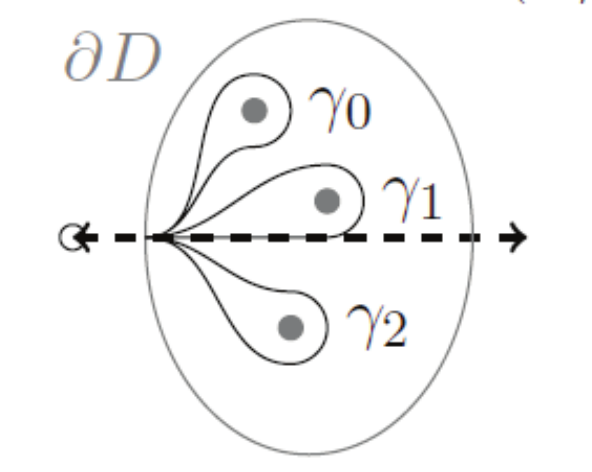}
\end{subfigure}
\begin{subfigure}[h]{.4\textwidth}
\centering
%\begin{tikzpicture}[x=1mm, y=1mm, inner xsep=0pt, inner ysep=0pt, outer xsep=0pt, outer ysep=0pt]
%
%\fill [fill=gray] (14,2) circle (.7mm);
%\draw (4,0) to[out=0,in=180] (12,-10) to[out=0,in=-90] (20,-2) to[out=90,in=0] (14,4) to[out=180,in=90] (12,2) to[out=-90,in=180] (14,0) to[out=0,in=90] (19,-2) to[out=-90,in=0] (12,-9) to[out=180,in=0] (4,0);
%\node[anchor={south}] at (18,3) {$\gamma^\prime_2$};
%\draw[->,thick,shorten <=1mm,shorten >=1mm, gray] (14,2) to[bend right=45] (12,-5);
%\draw[->,thick,shorten <=1mm,shorten >=1mm, gray] (12,-5) to[bend right=45] (14,2);
%
%
%\draw (0,0) circle (.7mm);
%\draw[line width=.3mm, dashed, <->] (0,0) -- (25,0);
%\draw[gray] (13,0) circle [x radius = 9, y radius = 12];
%\node[anchor={east}, gray] at (5,10) {$\partial D$};
%
%\foreach \x / \y / \i in {12/-5/1, 10/7/0}
%{
%
%	\fill [fill=gray] (\x,\y) circle (.7mm);
%	\draw (4,0) to[out=0,in=180] (\x,\y-2) to[out=0,in=270, ->] ++(2,2) to[out=90,in=0] ++(-2,2) to[out=180,in=0] (4,0);
%	\node[anchor={west}] at (\x+2,\y+1) {$\gamma^\prime_\i$};
%
%}
%\end{tikzpicture}
\includegraphics[scale=.4]{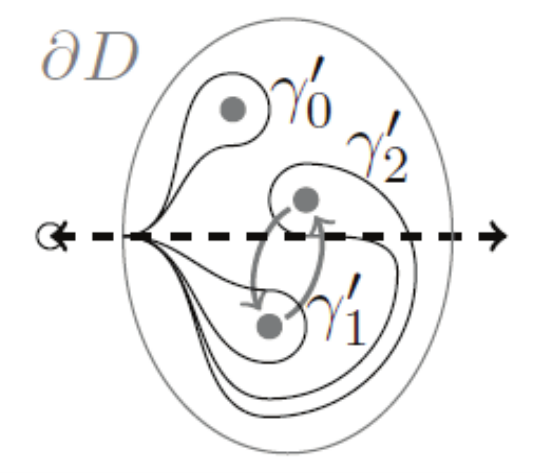}
\end{subfigure}

\subcaption{Loops $(\gamma_0,\gamma_1,\gamma_2)$ correspond to full exceptional collection $\sod{E_0,E_1,E_2}$. After acting by a generator of $B_3$, $\gamma_1^\prime = \gamma_2$. The corresponding full exceptional collection is the right mutation $\sod{E_0,E_2,R_{E_2} E_1}$. Note that $[\gamma_0 \circ \gamma_1 \circ \gamma_2] = [\gamma_0^\prime \circ \gamma_1^\prime \circ \gamma_2^\prime]$, consistent with the fact that the twists $T_{E_i}$ for any full exceptional collection compose to $\Phi_w$.}
\end{figure}

%%%%%		%%%%%		%%%%%
%%%%%		%%%%%		%%%%%
%%%%%		%%%%%		%%%%%
%%%%%		%%%%%		%%%%%

\subsection{Factoring spherical twists} \label{subsect:factorization}

The arguments used to establish fractional window shift autoequivalences extend to the general setting of Section \ref{sect:sphmut}. Suppose that $S:\cE \to \G$ is a spherical dg functor between pre-triangulated dg categories.  Assume that $\cE$ and $\G$ have generators and that $S$ and its adjoints are representable by bimodules.  Recall that since $S$ is a spherical functor, the functor 
$$ F_S = \op{Cone} \big( \op{id} \to RS \big) $$
is an equivalence.

We will now discuss a sufficient condition for a spherical twist to factor into a composition of other spherical twists.  In the following, angle brackets will be used to denote the category generated by a pair (tuple) of semiorthogonal subcategories of the ambient category as well as to assert that a given category admits a semiorthogonal decomposition.

\begin{thm}\label{thm:factorize}
Suppose that $\cE = \sod{\cA,\cB}$ and assume that the cotwist functor $F_S:\cE \to \cE$ has the property that there is a semiorthogonal decomposition
$$ \cE = \sod{ F_S(\cB), \cA }.$$
Then the restrictions $S_{\cA} = S|_{\cA}$ and $S_{\cB} = S|_{\cB}$ are spherical and 
$$ T_S = T_{S_{\cA}} \circ T_{S_{\cB}}. $$
\end{thm}

By Theorem \ref{thm:sph_fun_are_mutations} there exists a dg category $\cC$ such that
$$ \cC = \sod{\cE,\G} = \sod{\G,\cE'} = \sod{\cE',\G'} = \sod{\G',\cE} $$
where $S$, the spherical functor, is the composite $i^L_\G i_\cE$. We use the two mutation equivalences $\rmut_\G, \lmut_{\G^\prime}:\cE \to \cE'$ to induce decompositions $\cE' = \sod{\cA'_R,\cB'_R} = \sod{\rmut_\G(\cA), \rmut_\G (\cB)}$ and $\cE^\prime = \sod{\cA^\prime_L, \cB^\prime_L} := \sod{\lmut_{\G^\prime}(\cA), \lmut_{\G^\prime}(\cB)}$ respectively. Then due to the identity $F_S \simeq \rmut_{\G^\prime} \rmut_\G [1]$, the hypothesis in the statement of Theorem \ref{thm:factorize} is equivalent to the existence of a semiorthogonal decomposition
\begin{equation} \label{egn:factorization_criterion}
\cE^\prime = \sod{\cB^\prime_R, \cA^\prime_L}
\end{equation}
We will need the following
\begin{lem} \label{lem:factorization_lemma}
Under the hypothesis of Theorem \ref{thm:factorize}, $(\cA^\prime_L)^\perp = {}^\perp \cA$ and ${}^\perp (\cB^\prime_R) = \cB^\perp$ as subcategories of $\cC$.
\end{lem}
\begin{proof}
We deduce that $(\cA^\prime_L)^\perp = ^\perp\!\cA$ from the following sequence of mutations
\begin{equation} \label{eqn:factorize_SOD_1} \cC = \sod{\cA,\cB, \G} = \sod{\G, \cA^\prime_R, \cB^\prime_R} = \sod{\G,\cB^\prime_R,\cA^\prime_L} = \sod{\cB, \cG, \cA^\prime_L}.
\end{equation}
where the appearance of $\cA^\prime_L$ follows from \eqref{egn:factorization_criterion}. Similarly for ${}^\perp (\cB^\prime_R) = \cB^\perp$ we consider
\begin{equation} \label{eqn:factorize_SOD_2} \cC = \sod{\G^\prime, \cA,\cB} = \sod{\cA^\prime_L, \cB^\prime_L, \G^\prime} = \sod{\cB^\prime_R,\cA^\prime_L,\G^\prime} = \sod{\cB^\prime_R, \cG^\prime, \cA}.
\end{equation}
\end{proof}

\begin{proof} [Proof of Theorem \ref{thm:factorize}]
By assumption we have the semiorthogonal decomposition \ref{egn:factorization_criterion}, which implies that $\cB^\prime_R$ is left admissible and $\cA^\prime_L$ is right admissible in $\cC$. Furthermore Lemma \ref{lem:factorization_lemma} implies that $(\cA'_L)^\perp \cap {}^\perp\! \cB^\prime_R = {}^\perp\!\cA \cap \cB^\perp$, and we call this category $\cG_e$. Thus we have semiorthogonal decompositions
$$ \cC = \sod{\cB^\prime_R, \G_e, \cA'_L} = \sod{\cA, \G_e, \cB}. $$
In particular we have a semiorthogonal decomposition $\cB^\perp = ^\perp\!\cB^\prime_R = \sod{\G_e,\cA'_L}$

Combining this with the semiorthogonal decompositions \eqref{eqn:factorize_SOD_1} and \eqref{eqn:factorize_SOD_2} we obtain
\begin{gather*}
{}^\perp \cA = \sod{\cB, \cG} = \sod{\cG_e,\cB} = \sod{\cB_R^\prime, \cG_e} = \sod{\cG, \cB_R^\prime} , \text{ and}\\
\cB^\perp = \sod{\cA,\G_e} = \sod{\G^\prime,\cA} = \sod{\cA^\prime_L,\G^\prime} = \sod{\cG_e,\cA^\prime_L}.
\end{gather*}
An analogous analysis of $^\perp\!\cA = \sod{\cB,\G}$ gives the sequence of semiorthogonal decompositions.

Thus Theorem \ref{thm:mut_are_sph} implies that the functors $S_B := i^L_\cG i_\cB : \cB \to \cG$ and $\tilde{S}_\cA := i^L_{\cG_e} i_\cA : \cA \to \cG_e$ are spherical. Note that the left adjoints $i_\cG^L$ to the inclusions $i_\cG : \cG \to \sod{\cB,\cG}$ and to $i_{\cG_e} : \cC_{\cG_e} \to \sod{\cA,\cG_e}$ are the restrictions of the corresponding adjoints for the inclusions into $\cC$, so there is no ambiguity in writing $i^L_\cG$ and $i^L_{\cG_e}$ without further specification.

\begin{equation} \label{eqn:factorize_twist_diagram}
\xymatrix{ 
\cG \ar@/^/[rr]^{ \lmut_\cB = i^L_{\cG_e} i_{\cG} } \ar@{}[rr]|{ \circlearrowright T_{S_\cB} } & & \cG_e \ar@/^/[rr]^{\lmut_\cA = i^L_{\cG^\prime} i_{\cG_e}} \ar@/^/[ll]^{\lmut_{\cB_R^\prime} = i^L_\cG i_{\cG_e}} \ar@{}[rr]|{\circlearrowright T_{\tilde{S}_\cA}} & & \cG^\prime \ar@/^/[ll]^{\lmut_{\cA_L^\prime} = i^L_{\cG_e} i_{\cG'}} }
\end{equation}

Let $\phi:\G_e \to \G$ denote the isomorphism $i^L_{\cG} i_{\cG_e} = \lmut_{\cB_R^\prime}$ whose inverse is $\phi^{-1} = i^R_{\cG_e} i_{\cG}$. One checks that $S_\cA := i_\cG^L i_\cA$ is equivalent to $\phi \circ \tilde{S}_\cA : \cA \to \cG$ and is thus spherical, and $T_{S_\cA} \simeq \phi \circ T_{\tilde{S}_\cA} \circ \phi^{-1}$. Following the various isomorphisms in the diagram \eqref{eqn:factorize_twist_diagram} shows that
$$T_{S_\cA} \circ T_{S_\cB} = \phi \circ T_{\tilde{S}_\cA} \circ \phi^{-1} \circ T_{S_\cB} = \lmut_{\cB_R^\prime} \lmut_{\cA_L^\prime} \lmut_\cA \lmut_\cB = T_{S}$$
\end{proof}

\begin{rem}
Consider a relatively Calabi-Yau manifold $\pi:X \to \bP^d$.  The main result of \cite[Sec. 3]{AHK} is that, under certain assumptions on $\cO_X$ (namely that it is a spherical object on $X$ and EZ spherical with respect to $S=\pi$), the spherical twist associated to the functor $\pi^*$ factors into spherical twists around the objects $\pi^*\cO,\dotsc,\pi^*\cO(d)$.  The hypothesis that $\cO_X$ is EZ spherical implies that the cotwist $F_S = -\tensor \cO(-d-1)[d-n]$, which is the Serre functor up to a shift.  Hence the factorization of $T_S$ (called $\mathbf{H}$ in \cite{AHK}) is a consequence of Theorem \ref{thm:factorize}.
\end{rem}

\begin{ex}
Let $X$ be a smooth projective variety, and $j : Y \hookrightarrow X$ a smooth divisor. Then the restriction functor $S = j^\ast : \D^b(X) \to \D^b(Y)$ has a right adjoint $R = j_\ast$ and a left adjoint $L = j_\ast( \bullet \otimes \cO_Y(Y)[-1])$. The cotwist $F_S = \op{Cone}(\id \to j_\ast j^\ast) \simeq \bullet \otimes \cO_X(-Y)[1]$ is an equivalence, and $F_S L \simeq R$ by the projection formula. The corresponding spherical twist autoequivalence of $\D^b(Y)$ is
$$T_S (F^\bdot) = \op{Cone}(j^\ast(j_\ast F^\bdot)) \otimes \cO_Y(Y)[-1] \to F^\bdot) \simeq F^\bdot \otimes \cO_Y(Y)$$
In the special case where $Y$ is an anticanonical divisor, so that $F_S \simeq \bullet \otimes \omega_X[1]$. Then for any semiorthogonal decomposition $\D^b(X) = \sod{\cA,\cB}$ we have $\D^b(X) = \sod{F_S(\cB), \cA}$ by Serre duality, so Theorem \ref{thm:factorize} applies.
\end{ex}

\begin{ex}
An example studied in \cite{A} is that of a hypersurface $j: Y \hookrightarrow X$ where $\pi : Y \simeq \bP(E) \to M$ is a projective bundle of rank $r \geq 1$ over a smooth projective variety $M$. Then $j_\ast \pi^\ast : \D^b(M) \to \D^b(X)$ is spherical iff $\cO_Y(Y) \simeq \pi^\ast L \otimes \cO_\pi (-r)$. In this case the cotwist is tensoring by a shift of $L$, so if $L \simeq \omega_M$, then Theorem \ref{thm:factorize} applies to any semiorthogonal decomposition $\D^b(M) = \sod{\cA,\cB}$
\end{ex}

%%%%%		%%%%%		%%%%%
%%%%%		%%%%%		%%%%%
%%%%%		%%%%%		%%%%%
%%%%%		%%%%%		%%%%%

\subsection{Autoequivalences of complete intersections}
\label{subsect:complete_intersections}
Suppose $X_s \subset X$ is defined by the vanishing of a regular section $s$ of a vector bundle $\cV^\vee$.  In this section, we will use a standard construction to produce autoequivalences of $\D(X_s)$ from variations of GIT for the total space of $\cV$.  This forms a counterpart to \cite[Sections 4,5]{BFK}, where equivalences between different complete intersections are considered.

We are interested in the case where the total space of $\cV$ is Calabi-Yau.  If $X= \bP^n$ and $\cV$ is completely decomposable, this is equivalent to $X_s$ being Calabi-Yau.  Since $X_s$ is defined by a regular section, the Koszul complex $(\wedge^\bullet \cV, d_s)$ is a resolution of $\cO_{X_s}$. The key ingredient in this discussion is an equivalence of categories between $\D(X_s)$ and a category of generalized graded matrix factorizations associated to the pair $(\cV,s)$.

We call the data $(\fX,W)$ where $\fX$ is a stack equipped with a $\bC^*$ action factoring through the squaring map and $W$ is a regular function of weight 2 a \emph{Landau-Ginzburg (LG) pair}.  Let $\pi:\cV \to X$ be the vector bundle structure map.  There is an obvious action of $\bC^*$ on $\cV$ by scaling along the fibers of $\pi$.  We equip $\cV$ instead with the square of this action, so that $\lambda$ acts as scaling by $\lambda^2$.  Since $s$ is a section of $\cV^\vee$, it defines a regular function $W$ on $\cV$ that is linear along the fibers of $\pi$.  By construction it has weight 2 for the $\bC^*$ action.  The total space of $\cV|_{X_s}$ is $\bC^*$-invariant and when we equip $X_s$ with the trivial $\bC^*$ action we obtain a diagram 
$$ \xymatrix{ \cV|_{X_s} \ar[r]^i \ar[d]^\pi & \cV \\ X_s & } $$
of LG pairs where the potentials on $\cV|_{X_s}$ and $X_s$ are zero.

The category of curved coherent sheaves on an LG pair $\D(\fX,W)$ is the category whose objects are $\bC^*$-equivariant coherent sheaves $\cF$ equipped with an endomorphism $d$ of weight 1 such that $d^2 = W \cdot \id$; and whose morphisms are obtained by a certain localization procedure.  The maps in the above diagram induce functors 
$$ \xymatrix{ \D(X_s) = \D(X_s,0) \ar[r]^(.6){\pi^*} & \D(\cV|_{X_s},0) \ar[r]^{i_*} & \D(\cV,W) } $$
whose composite $i_* \pi^*$ is an equivalence. (For details, see \cite{Is,S12}.)

Suppose that $V$ is a smooth quasiprojective variety with an action of a reductive algebraic group $G \times \bC^*$, where $\bC^*$ acts through the squaring map.  Let $W$ be a regular function on $V$ which is $G$ invariant and has weight 2 for $\bC^*$.  Suppose that $\cL$ is a $G \times \bC^*$ equivariant line bundle so that $(\cV,W) \cong (V^{ss}(\cL)/G, W)$ equivariantly for the $\bC^*$ action.  For simplicity assume that $V^u(\cL)$ consists of a single KN stratum $S$ with 1 PSG $\lambda$.  Let $Z$ be the fixed set for $\lambda$ on this stratum and $Y$ its blade.  Write $\sigma:Z \to V$ for the inclusion.  As above we define full subcategories of $\D(V/G,W)$.  Let $\G_w \subset \D(V/G,W)$ be the full subcategory of objects isomorphic to objects of the form $(\cE,d)$ where $\sigma^*\cE$ has $\lambda$-weights in $[w,w+\eta)$.  We also define the larger subcategory $\cC_w$ where the weights lie in $[w,w+\eta]$.  The analysis for the derived category can be adapted to the category of curved coherent sheaves \cite{BFK} and we see that $\G_w$ is admissible in $\cC_w$.  The maps $i:Y \to V$ and $p:Y \to Z$ induce functors $p^*:\D(Z/L,W|_Z) \to \D(Y/P,W|_Y)$ and $i_*:\D(Y/P,W|_Y) \to \D(V/G,W)$.  Let $\D(Z/L,W|_Z)_w$ be the full subcategory of curved coherent sheaves concentrated in $\lambda$-weight $w$.  Then $i_*p^*:\D(Z/L,W|_Z)_w \to \D(V/G,W)$ is fully faithful and has image $\cA_w$.

We now consider a balanced wall crossing which exchanges $\lambda = \lambda^+$ with $\lambda^{-1} = \lambda^-$ and $S = S^+$ for $S^-$.  Then we obtain wall crossing equivalences.  Since $\cC_w$ and $\G_w$ are defined by weight conditions, as above we see that $\cC^+_w = \cC^-_{-w - \eta}$ and $\G^-_w$ is the left orthogonal to $\cA_w$.  Therefore, the window shift autoequivalence in this context is still realized by a mutation.

\begin{ex}\label{ex:K3}
We consider a K3 surface $X$ obtained as a complete intersection of type $(2,0),(1,3)$ in $X=\bP^2 \times \bP^2$.  It is well known that line bundles on a K3 surface are spherical.  We will see that the window shift automorphisms of $\D(X)$ coming from VGIT as above are the compositions of spherical twists around $\cO_{X}(i,0)$ then $\cO_{X}(i+1,0)$ or around $\cO_{X}(0,i),\cO_{X}(0,i+1),$ and $\cO_{X}(0,i+2)$.  

Let $\cV = \cO(-2,0) \oplus \cO(-1,-3)$.  Recall that the total space of $\cV$ is a toric variety which can be obtained as a GIT quotient of $\bA^8$ by $(\bC^*)^2$ under the action
$$ (t,s) \mapsto (t,t,t,s,s,s,t^{-2},t^{-1}s^{-3}). $$
We also recall that the wall and chamber decomposition of $\bR^2$ associated to this action is given in the following diagram.
\begin{figure}[h]
\centering
\begin{tikzpicture}[x=.5mm, y=.5mm, inner xsep=0pt, inner ysep=0pt, outer xsep=0pt, outer ysep=0pt]

% create grid & axes
\draw[style={gray, very thin},step=10] (-40,-40) grid (40,40);

% walls
\draw[line width=.5mm] (0,0)--(-40,0);
\draw[line width=.5mm] (0,0)--(0,-40);
\draw[line width=.5mm] (0,0)--(40,0);
\draw[line width=.5mm] (0,0)--(13.33,40);

\node at (-45,0) {$W_2$};
\node at (0,-45) {$W_1$};

% chambers
\node at (-35,-35) {I};
\node at (-35,35) {II};
\node at (35,35) {III};
\node at (35,-35) {IV};

\end{tikzpicture}
\end{figure}

Chamber I corresponds to $\op{tot} \cV$ and we will analyze the autoequivalences of $X_s$ that come from the walls $W_1$ and $W_2$.  The window shift autoequivalences of $\bD^b(\op{tot} \cV)$ coming from $W_1$ do not factor because the associated $Z/L$ is not compact.  However, in the presence of a potential, $Z/L$ becomes compact.  In fact, the associated Landau-Ginzburg model actually admits a full exceptional collection.  To proceed we must compute the KN stratifications near the walls.  Write $V_{\bullet}$ for the locus defined by the vanishing of the variables occurring in $\bullet$.  (So $V_x$ is the locus where all of the $x_i$ are zero.) We obtain the table below in which we indicate the ordering of the strata, the attached destabilizing one parameter subgroup, its fixed set and the entire stratum.  

\begin{table}[h]
\begin{tabular}{c|c} 
\multicolumn{2}{c}{Near $W_1$} \\ 
Chamber I & Chamber IV \\ \hline
\begin{tabular}{c|c|c|c}
& \text{1-PSG} & \text{fixed set} & \text{stratum} \\ \hline
\text{strat. 1} & $\lambda_0$ & 0 & $V_{xy}$\\
\text{strat. 2} & $(0,-1)$ & $V_{yq}\setminus V_x $ & $V_y \setminus V_x$ \\
\text{strat. 3} & $(-1,0)$ & $V_{xpq}\setminus V_{xy}$ & $V_x \setminus V_{xy}$ \\
\end{tabular} 
&
\begin{tabular}{c|c|c}
\text{1-PSG} & \text{fixed set} & \text{stratum} \\ \hline
$\lambda_0$ & 0 & $V_{yp}$\\
$(0,-1)$ & $V_{yq}\setminus V_x $ & $V_y \setminus V_x$ \\
$(1,0)$ & $V_{xpq}\setminus V_{xy}$ & $V_{pq} \setminus V_y$ \\
\end{tabular} 
\\ 
\vspace{1cm} &  \\
\multicolumn{2}{c}{Near $W_2$} \\  
Chamber I & Chamber II  \\ \hline
\begin{tabular}{c|c|c|c}
& \text{1-PSG} & \text{fixed set} & \text{stratum} \\ \hline
\text{strat. 1} & $\lambda_0$ & 0 & $V_{xy}$\\
\text{strat. 2} & $(-1,0)$ & $V_{xpq}\setminus V_{xy}$ & $V_x \setminus V_{xy}$ \\
\text{strat. 3} & $(0,-1)$ & $V_{yq}\setminus V_x $ & $V_y \setminus V_x$ \\
\end{tabular}
&
\begin{tabular}{c|c|c}
\text{1-PSG} & \text{fixed set} & \text{stratum} \\ \hline
$\lambda_0$ & 0 & $V_x$ \\
$(-1,0)$ & $V_{xpq}\setminus V_{xy}$ & $V_x \setminus V_{xy}$ \\
$(0,1)$ & $V_{yq} \setminus V_x $ & $V_q \setminus V_x$ \\
\end{tabular} \\
\vspace{1cm}
\end{tabular} 
\caption{The Kirwan-Ness stratification for $T$ acting on $\bA^8$}
\end{table}

Consider the potential $W = p f + g q \in \bC[x_i,y_i,p,q]_{i=0}^2$, where $f \in \bC[x_i]$ is homogeneous of degree 2 and $g \in \bC[x_i,y_i]$ is homogeneous of degree $(1,3)$.  In order to define an LG pair, we must also specify a second grading on $\bC[x_i,y_i,p,q]$.  We define the LG weights of $p$ and $q$ to be 2 and the LG weights of $x_i$ and $y_i$ to be 0.  Assume that $f$ defines a smooth rational curve in $\bP^2$.  In order to proceed, we need to introduce a particular type of curved coherent sheaf.  Consider a line bundle $\cL$ on an LG pair which is equivariant for the $\bC^*$ action.  Given sections $a \in \Gamma(\cL)$ and $b \in \Gamma(\cL^\vee)$ of weight 1, we form a curved coherent sheaf for the potential $b(a)$:
$$ \xymatrix{ \cO \ar@/^/[r]^a & \cL \ar@/^/[l]^b, } \quad \text{i.e.} \quad d = \begin{pmatrix} 0 & b \\ a & 0 \end{pmatrix}, $$
and denote it by $\{a,b\}$.  We also write $\cO_{triv} = \{1,W\}$ (where $1,W$ are weight 1 section and co-section of $\cO(-1)_{LG}$).  This object is isomorphic to zero in the category of curved coherent sheaves.

Let us analyze what happens near $W_1$.  First, we have computed that for the least unstable stratum 
$$ Z_1/L_1= (V_{xpq} \setminus V_{xy}) / T \cong \bP^2 / \bC^*. $$
Next, we notice that $W|_{Z_1} = 0$ and that $Z_1$ is contained in the fixed set for the LG $\bC^*$ action.  Therefore the category $\D(Z_1/L_1,W|_{Z_1}) \cong \D(\bP^2/\bC^*)$ and for any $w$ we have $\D(\bP^2/\bC^*)_w \cong \D(\bP^2)$.  It is well known that $\D(\bP^2)$ admits a full exceptional collection of length 3.  For example $\D(\bP^2) = \sod{ \cO,\cO(1),\cO(2) }$.  By the curved analog of Proposition \ref{prop:window_shift_formula}, we compute the spherical object associated to $\cO(i)$ on $\bP^2$ by pulling it back to $V_{pq} \setminus V_y$, pushing it forward to $V \setminus V_y$, then restricting it to $\cV = (V \setminus V_x \cup V_y )/T$.  The locus $V_{pq}$ restricts to the zero section of $\cV$, which we also denote by $X$.  The object corresponding to $\cO(i)$ on $Z_1/L_1$ is the line bundle $\cO_X(0,i)$, viewed as a curved coherent sheaf supported on the zero section.  This object corresponds to an object of $\D(X)$.  To compute this object we observe that there are short exact sequences
$$
\xymatrix{ 
0 \ar[r] & \cO_{triv} \otimes \{q,g\} \ar[r] & \{p,f\} \otimes \{q,g\} \ar[r] & \cO_{p=0} \otimes \{q,g\}  \ar[r] & 0 \\
0 \ar[r] & (\cO_{p=0})_{triv} \ar[r] & \cO_{p=0} \otimes \{q,g\} \ar[r] & \cO_S \ar[r] & 0 
}
$$
This implies that $\cO_X(0,i)$ is equivalent to $\{p,f\} \otimes \{q,g\}\otimes \cO(0,i)$ in $\D(\cV,W)$.  Using the analogous short exact sequences for $f$ and $g$ we see that it is also equivalent to $\cO_{\cV|_Y}(0,i)$.  However, this is the image of $\cO_{X}(0,i)$ under the equivalence $\D(X) \cong \D(\cV,W)$.

Next, we consider the wall $W_2$.  In this case, we have
$$ Z_2/L_2 = (V_{yq} \setminus V_x ) / T \cong ( \op{tot} \cO_{\cP^2} (-2) )/ \bC^*. $$
Moreover $W|_{Z_2} = pf$.  So we have $\D(Z_2/L_2,W|_{Z_2}) \cong \D(C/\bC^*)$, where $C\subset\bP^2$ is the rational curve defined by $f$.  This means that for any fixed $w$, $\D(Z_2/L_2,W|_{Z_2})_w \cong \D(\bP^1)$.  Of course, we have $\D(\bP^1) = \sod{\cO,\cO(1)}$.  We play a similar game to compute the objects in $\D(Y)$ corresponding to these line bundles.  First, $\cO_C(i)$ corresponds to the curved coherent sheaf $\cO_{\op{tot} \cO_C(-2)}(i)$ on $Z_2/L_2$.  We push this forward and restrict to $\cV$ to get the line bundle $\cO(i,0)$ on the locus $\{q=f=0\}$.  By considering short exact sequences as in the previous case, we see that these objects correspond to the objects $\cO_{X}(i,0)$ in $\D(X)$.
\end{ex}

%%%%%		%%%%%		%%%%%
%%%%%		%%%%%		%%%%%
%%%%%		%%%%%		%%%%%
%%%%%		%%%%%		%%%%%

\end{document}